\documentclass[11pt,reqno]{amsart}

\usepackage{hyperref}

\allowdisplaybreaks[4]
\usepackage{amsmath}
\usepackage{amssymb}
\usepackage{mathrsfs}
\usepackage{amsthm}
\usepackage{bm}
\usepackage{graphicx}
\usepackage{booktabs}
\usepackage{float}
\usepackage{geometry}
\usepackage{color}

\theoremstyle{plain}
\newtheorem{lemma}{Lemma}[section]
\newtheorem{thm}[lemma]{Theorem}
\newtheorem{dfn}[lemma]{Definition}
\newtheorem{remark}[lemma]{Remark}
\newtheorem{prop}[lemma]{Proposition}
\newtheorem{assumption}{Assumption}
\begin{document}
\title[Weak intermittency and second moment bound of a fully discrete scheme]{Weak intermittency and second moment bound of a fully discrete scheme for stochastic heat equation}
\author{Chuchu Chen, Tonghe Dang, Jialin Hong}
\address{LSEC, ICMSEC,  Academy of Mathematics and Systems Science, Chinese Academy of Sciences, Beijing 100190, China,
\and 
School of Mathematical Sciences, University of Chinese Academy of Sciences, Beijing 100049, China}
\email{chenchuchu@lsec.cc.ac.cn; dangth@lsec.cc.ac.cn; hjl@lsec.cc.ac.cn}
\thanks{This work is funded by National Natural Science Foundation of China (No. 11971470,
No. 11871068, No. 12031020, No. 12022118)}
\begin{abstract} In this paper, we first prove the weak intermittency, and in particular the sharp exponential order $C\lambda^4t$ of the second moment of the exact solution of the stochastic heat equation with multiplicative noise and periodic boundary condition, where $\lambda>0$ denotes the level of the noise. 
In order to inherit numerically these intrinsic properties of the original equation, we introduce a fully discrete scheme, whose spatial direction is based on the finite difference method and temporal direction is based on the $\theta$-scheme. 
We prove that the second moment of numerical solutions of both spatially semi-discrete and fully discrete schemes grows at least as $\exp\{C\lambda^2t\}$ and at most as $\exp\{C\lambda^4t\}$ for large $t$ under natural conditions, which implies the weak intermittency of these numerical solutions. Moreover, a renewal approach is applied to show that both of the numerical schemes could preserve the sharp exponential order $C\lambda^4t$ of the second moment of the exact solution for large spatial partition number.
\end{abstract}
\keywords {Stochastic heat equation $\cdot$ Weak intermittency $\cdot$ Sharp exponential order $\cdot$ Numerical scheme $\cdot$ Discrete Green function}

\maketitle
\section{Introduction}

In this paper, we study numerically the preservation of the weak intermittency and the sharp exponential order of the second moment of the exact solution of the following stochastic heat equation (SHE) with periodic boundary condition: 
\begin{align}
\begin{cases}
\partial_t u(t,x)=\partial^2_x u(t,x)+\lambda\sigma(u(t,x))\dot{W}(t,x),\quad &(1.1a)\vspace{1 ex}\\
u(t,0)=u(t,1),\quad t\ge 0,\quad &(1.1b) \vspace{1 ex}\\
u(0,x)=u_0(x),\quad 0\leq x\leq 1.\quad &(1.1c)
\end{cases}\label{she}
 \end{align}
Here, $\dot{W}(t,x), t\ge0,x\in[0,1]$ denotes the space-time white noise with respect to some given filtered probability space $\big(\Omega,\mathcal{F}, \left \{\mathcal{F}_t\right \}_{t\ge 0},\mathbb{P}\big)$, $\lambda>0$ denotes the level of the noise, $\sigma: \mathbb{R} \rightarrow \mathbb{R}$ is a globally Lipschitz function, and $u_0$ is a bounded, non-negative, non-random and measurable function. Eq. $(1.1a)$ characterizing the evolution of a field in a random media, arises in several settings, for example generalized Edwards-Wilkinson models for the roughening of surfaces, continuum limits of particle processes and
continuous space parabolic Anderson models (PAMs) (see \cite{setting07} and references therein). 

For a random field of multiplicative type, intermittency is a universal phenomenon (see \cite{ZRS90}). It originates from the physics literature on turbulence and refers to the chaotic behavior of a random field that develops unusually high peaks over small areas (see \cite{dynamo93, feynman95}).
Recall that the upper $p$th moment Lyapunov exponent of $u$ at $x$ is defined as $\bar{\gamma}_p(x):=\limsup_{t\rightarrow\infty}\frac{\log\left|u(t,x)\right|^p}{t},\forall p>0.$ The random field $u$ is called intermittent (also called fully intermittent) if for all $x,$ the mapping $p\rightarrow \bar{\gamma}_p(x)/p$ is strictly increasing on $p\in[1,\infty).$ This mathematical definition implies that the appearance of high peaks giving the main contribution to the statistical moments of the solution, leads to the non-trivial exponential behaviors of the moments of the solution. The existing research on the intermittency usually begins with a Feynmann-Kac type formula to calculate the explicit expression of the $p$th moment Lyapunov exponent of the solution. For example, in the case of $\sigma(u)=u,$ which refers to the famous PAM, it is shown in \cite{anderson94, chenxia15} that the solution of PAM is intermittent both in the continuous case and in the spatially discrete case. In the nonlinear case, it is difficult to obtain the explicit expression of the $p$th moment Lyapunov exponent, so there comes a notion called weak intermittency, which means for all $x$, $\bar{\gamma}_2(x)>0$ and $\bar{\gamma}_p(x)<\infty$  $(\forall p\ge 2)$. It is shown in \cite{KF09} that the weak intermittency implies intermittency whenever the comparison principle holds.

For Eq. $(1.1a)$ in the whole space, the weak intermittency of the solution both on the real line (\cite{KF09}) and on the lattice (\cite{chenle17, FD12, GJK15}) has been studied. For the continuous Eq. $(1.1a)$ with various boundary conditions in a bounded domain, the weak intermittency and in particular the effects of the noise level $\lambda$ on the second moment of the solution have been extensively studied (see \cite{foondun14, KDK15, dissipation20, xie16}). More precisely, for the case with Dirichlet boundary condition, it is proved in \cite{foondun14, xie16} that the second moment of the solution grows exponentially fast if the noise level $\lambda$ is large enough, and decays exponentially if $\lambda$ is small. While for the case with Neumann boundary condition, it is shown in \cite{foondun14} that the second moment of the solution grows exponentially fast no matter what $\lambda$ is. A fine result is proved in \cite{KDK15}, which suggests that the second moment of the solution has sharp exponential order $C\lambda^4t$. As for the case with periodic boundary condition, based on the analysis of the Green function of \eqref{she}, it is proved in Section \ref{sec2} that \eqref{she} is weakly intermittent and the second moment of the solution has the sharp exponential order $C\lambda^4t.$

In general, the solutions of stochastic partial differential equations can not be solved exactly, thus numerical methods provide a qualitative and quantitative approach to investigate the properties of the exact solution, which have been developed in the past three decades. Many spatially and temporally discrete schemes, for instance the finite difference method, the finite element method, explicit and implicit Euler method and exponential integrator method, have been well studied (see e.g. \cite{Cohen20, CHL17, CHJ19, CHJ192, CCHS20, CH19, CH20, CHL172, CHS21, Gy98, Gy99, Walsh05}).
It is natural to ask:
\begin{itemize}
\item[$(Q_1)$] Is there a numerical scheme inheriting the weak intermittency of \eqref{she}? 
\item[$(Q_2)$] Furthermore, can the numerical scheme preserve the sharp exponential order of the second moment of the exact solution?
\end{itemize}

Considering the above two questions, we apply the finite difference method to \eqref{she} to obtain a spatially semi-discrete scheme, which is convergent to the exact solution in the mean square sense with order $1/2$. By finding the explicit expression of the semi-discrete Green function, the continuous version solution of the semi-discrete scheme can be written into a compact form, which plays a key role in the analysis of the weak intermittency. 
With the detailed analysis on the integral properties of the semi-discrete Green function, the a priori estimation of the numerical solution gives an intermittent upper bound ($Cp^3\lambda^4$) for the upper $p$th moment Lyapunov exponent. The semi-discrete Green function, whose point-wise property is slightly different from the continuous one, is proved to be positive when time is large. This positivity, combining with a modified reverse Gr{\"o}nwall's inequality reveals an intermittent lower bound ($C\lambda^2$) of the upper $2$th moment Lyapunov exponent under natural conditions. These imply that the numerical solution of this semi-discretization is weakly intermittent. 
To enhance the exponential order of the second moment of the semi-discrete numerical solution, a renewal approach depending on the finer integral lower estimate of the semi-discrete Green function on the spatial grid points is applied. We prove that the second moment of the semi-discrete numerical solution on the spatial grid points has sharp exponential order $C\lambda^4t$ provided additionally that the initial data is a positive constant and the partition number is large.
%$n\ge \zeta\lambda^2$ with a positive constant $\zeta$, the second moment of the numerical solution has sharp exponential order $C\lambda^4t$. 
%At last, we show the error estimation between the semi-discrete Green function and the exact one, which leads to the convergence in the mean square sense of the finite difference method.

For the full discretization, we apply further the $\theta$-scheme to get a fully discrete scheme, which is convergent to the exact solution in the mean square sense with order $1/2$ in the spatial direction and $1/4$ in the temporal direction. The compact integral form is formulated by presenting the explicit expressions of the fully discrete Green functions. The prerequisite for the proof of the weak intermittency is the technical estimates of the fully discrete Green functions. 
%{\color{red}Due to the stability requirements of space and time step sizes,} the estimates are given under some Courant-Friedrichs-Lewy conditions when necessary. 
We prove that the numerical solution of this fully discrete scheme is weakly intermittent with an intermittent upper bound ($Cp^3\lambda^4$) for the upper $p$th moment Lyapunov exponent and an intermittent lower bound ($[\log(1+C\lambda^2\tau)]/\tau$ with $\tau$ being the time step size) for the upper $2$th moment Lyapunov exponent. This implies that the second moment of the numerical solution of the fully discrete scheme grows at most as $\exp\{C\lambda^4t_m\}$ and at least as $\exp\{C\lambda^2t_m\}$ for sufficiently large $t_m:=m\tau$. To fill the gap of the index of $\lambda$, a discrete renewal method is implemented, which essentially depends on the finer estimate of the fully discrete Green function. Under some coupling condition between the space and time step sizes, we prove that the second moment of this fully discrete scheme has sharp exponential order $C\lambda^4t_m$.

This paper is organized as follows. In Section \ref{sec2}, the weak intermittency of the mild solution of \eqref{she} is established. In Section \ref{sec3}, we apply the finite difference method to \eqref{she} for spatial discretization, and prove the weak intermittency and the sharp exponential order of the second moment of the numerical solution of this spatially semi-discrete scheme. The convergence order in the mean square sense of the spatially semi-discrete scheme is given.
%The error estimations between the semi-discrete Green function and the exact one, together with the convergence order of the spatially semi-discrete scheme to \eqref{she} are given. 
Section \ref{sec4} is devoted to the analysis of the fully discrete scheme on the preservation of the weak intermittency and the sharp exponential order of the second moment of the exact solution. Moreover, we give the mean square convergence order of the fully discrete scheme.
%Green functions as well as the convergence order of the fully discrete scheme to \eqref{she} under some suitable conditions. 
In Section \ref{conclusion}, we give our conclusions and propose several open problems for future study. At last, some proofs are given in the appendix.
%We point out that constants $C$ in this paper may be different from line to line.

\section{Weak intermittency of exact solution}\label{sec2}
The goal of this section is to investigate the weak intermittency of the mild solution of \eqref{she}. 
%The upper and lower bounds are given in Section \ref{sec2.1} and Section \ref{sec2.2} respectively.  
The goal of this section is to investigate the weak intermittency of the mild solution of \eqref{she}. Before that, we first present the definitions of Lyapunov exponent and intermittency, which can be found in \cite{KD14}. Throughout this paper, we let $\textbf{i}^2=-1,$ and constant $C$ may be different from line to line.%Then based on the expansions of Green functions to \eqref{she} with Neumann boundary and periodic boundary, we show some properties of them  respectively. After that, we establish the a priori estimations of moments for mild solution to \eqref{she} with these boundary conditions, which lead to the upper bound for $p$th moment Lyapunov exponent. Finally, we prove that solutions under these two boundary conditions are both weakly intermittent.\\

\begin{dfn}\label{def2.1}
Fix some $x\in[0,1]$, define the upper $p$th moment Lyapunov exponent of $u$ at $x$ as 
 \begin{equation}\label{lyapunov}
 \bar{\gamma}_p(x):=\limsup_{t\rightarrow \infty}\frac{1}{t}\log \mathbb{E}\left(\left|u(t,x)\right|^p\right),
 \end{equation}
 for all $p\in (0,\infty)$.
 \end{dfn}

\begin{dfn}\label{def2.2}
 (\romannumeral1) We say that $u$ is fully intermittent if for all $x\in[0,1]$, the map 
 $
 p\rightarrow \frac{\bar{\gamma}_p(x)}{p}
 $
 is strictly increasing for $p\in[2,\infty)$.\\
 (\romannumeral2) We say that $u$ is weakly intermittent if for all $x\in[0,1]$, $\bar{\gamma}_2(x)>0$ and $\bar{\gamma}_p(x)<\infty$ for each $p>2$. 
 \end{dfn}
 \begin{remark}\label{remark2.3}
 (\romannumeral1) The full intermittency can be implied by the weak intermittency on some certain circumstances, for example, $\sigma(0)=0$ and $u_0(x)\ge 0$. For its proof, we refer to \cite[Theorem 3.1.2]{anderson94}.\\
 (\romannumeral2) All the results in this paper are still valid if we choose the lower $p$th moment Lyapunov exponent, whose definition is
 $$\underline{\gamma}_p(x):=\liminf_{t\rightarrow \infty}\frac{1}{t}\log \mathbb{E}(\left|u(t,x)\right|^p).$$
 \end{remark}
The mild solution of \eqref{she} can be written as
\begin{align}\label{mild periodic}
u(t,x)=\int_{0}^{1}G(t,x,y)u_0(y)\,dy+\lambda\int_{0}^{t}\int_{0}^{1} G(t-s,x,y)\sigma\left(u(s,y)\right)\,dW(s,y),
\end{align}
where the Green function $G(t,x,y)$ is defined as (see \cite{Green01})
 \begin{align}\label{greenp1}
G(t,x,y)=\frac{1}{\sqrt{4\pi t}}\sum _{m=-\infty}^{+\infty}e^{-\frac{(x-y-m)^2}{4t}},\quad t>0,\quad x,y\in[0,1],
\end{align}
and its spectral decomposition is
\begin{align}\label{greenp2}
G(t,x,y)=\sum_{j=-\infty}^{+\infty}e^{-4\pi^2j^2t}e^{2\pi \textbf{i}j(x-y)},\quad t>0,\quad x,y\in[0,1].
\end{align}

In order to investigate the weak intermittency of the exact solution of \eqref{she}, we make the following assumption on the initial data and diffusion coefficient.
\begin{assumption}\label{Assumption 2}
Let $I_0:=\inf_{x\in[0,1]}u_0(x)$, and $J_0:=\inf_{x\in \mathbb{R}\backslash\{0\}}\left|\frac{\sigma (x)}{x}\right|$. We assume that $I_0>0, J_0>0$.
\end{assumption}
 \begin{thm}\label{thm2.4}
 Under Assumption \ref{Assumption 2}, the solution of \eqref{she} is weakly intermittent.
 \end{thm}

Before giving the proof of Theorem \ref{thm2.4}, we show some properties of the Green function to \eqref{she}.
\begin{lemma}\label{lemma2.5}
$G(t,x,y)$ has the following properties:\vspace{1 ex}\\
(\romannumeral1) $G(t,x,y)\ge 0$ for  $t> 0,x,y\in[0,1]$, and $\int_{0}^{1}G(t,x,y)\,dy=1$ for $t> 0,x\in[0,1]$.\vspace{1 ex}\\
(\romannumeral2) $\int_{0}^{1}G^2(t,x,y)\,dy=G(2t,x,x)\ge \frac{1}{\sqrt{8\pi t}}$ for $t>0,x\in[0,1]$.\vspace{1 ex}\\
(\romannumeral3) $\int_{0}^{1}G^2(t,x,y)\,dy\leq C\left(\frac{1}{\sqrt{t}}+1\right)$ with a positive constant $C$ for all $t>0,x\in[0,1]$.
\end{lemma}
\begin{proof}
It is obvious that $(\romannumeral1)$ holds. We prove $(\romannumeral2)$ by the use of the spectral decomposition \eqref{greenp2} of the Green function.
\begin{equation*}
\begin{aligned}
\int_{0}^{1}G^2(t,x,y)\,dy&=\int_{0}^{1}\sum_{r,j=-\infty}^{+\infty}e^{-4\pi^2 (r^2+j^2)t}e^{2\pi \textbf{i}(r+j)(x-y)}\,dy=\sum_{\left\{r,j\in\mathbb{Z};\; r+j=0\right\}}e^{-4\pi^2(r^2+j^2)t}\\
&=\sum_{r=-\infty}^{+\infty}e^{-8\pi^2 r^2t}=G(2t,x,x).
\end{aligned}
\end{equation*}
By \eqref{greenp1}, we have
\begin{align*}
G(2t,x,x)=\frac{1}{\sqrt{8\pi t}}\sum_{m=-\infty}^{+\infty}e^{-\frac{m^2}{8t}}\ge \frac{1}{\sqrt{8\pi t}}.
\end{align*}
As for $(\romannumeral3)$, combining $(\romannumeral2)$ and \cite[Lemma B.1]{dissipation20}, we can get the desired result.
The proof is finished.
\end{proof}
\noindent
\textbf{Proof of Theorem \ref{thm2.4}:}
\begin{proof}
 The following intermittent upper bound is a direct consequence of \cite[Proposition 4.1]{dissipation20}: 
\begin{align*}
\sup_{x\in[0,1]}\bar{\gamma}_p(x)\leq CL_{\sigma}^4\lambda^4p^3,
\end{align*}
with some constant $C>0$ for all $p\in[2,\infty),$ where $L_{\sigma}$ is the Lipschitz constant of $\sigma$.

Following the approach presented by Khoshnevisan et al. in \cite[Section 2.2]{KDK15}, and combining Lemma \ref{lemma2.5} $(\romannumeral1)$ $(\romannumeral2)$, the intermittent lower bound is given below.
 Under Assumption \ref{Assumption 2}, $\inf_{x\in[0,1]}\bar{\gamma}_2(x)\ge \frac{\lambda^4J^4_0}{8}>0$. Hence, the proof is finished.
 \end{proof}
 At the end of this subsection, let's intuitively see the information that the weak intermittency can bring to us.
 Suppose $\bar{\gamma}_2(x)=\underline{\gamma}_2(x):=\gamma_2(x).$
 Take constants $\alpha_1,\alpha_2$ satisfying $$0<\alpha_1\lambda^4<C_1\lambda^4\leq\frac{\gamma_2(x)}{2}\leq C_2\lambda^4<\alpha_2\lambda^4.$$ Set $B_1(t):=\left\{\omega\in\Omega:\left|u(t,x)(\omega)\right|>e^{\alpha_2\lambda^4t}\right\}$ and $B_2(t):=\left\{\omega\in\Omega:\left|u(t,x)(\omega)\right|<e^{\alpha_1\lambda^4t}\right\}$.
 
 By Chebyshev's inequality, 
 $$
 \mathbb{P}(B_1(t))\leq e^{-2\alpha_2\lambda^4t}\mathbb{E}\left|u(t,x)\right|^2\approx e^{-(2\alpha_2\lambda^4-\gamma_2(x))t}\leq e^{-Ct}$$ with some $C>0,$ where $f(t)\approx g(t)$ means $\lim_{t\rightarrow\infty}(\log f(t)-\log g(t))/t=0.$ This implies that the random field $u$ may take very large values with exponentially small probabilities, and therefore it develops high peaks when $t$ is large.
 
  Moreover, $$\mathbb{E}\left(|u(t,x)|^2;B_2(t)\right)\leq e^{2\alpha_1\lambda^4t}\ll e^{\gamma_2(x)t}\approx \mathbb{E}|u(t,x)|^2,$$ where $f(t)\ll g(t)$ denotes $\lim_{t\rightarrow\infty}f(t)/g(t)=0.$ This means the contribution to the second moment of $u$ at $x$ comes from $(B_2(t))^c$ where may appear the high peak for large $t$.
  
   When the random field $u$ is fully intermittent, the main contribution to each moment of $u$ is carried by higher and higher, more and more widely spaced peaks. The above analysis is also valid for numerical solution.
For more details, we refer to \cite{feynman95, dissipation20}.

\section{Intrinsic property-preserving spatial semi-discretization}\label{sec3}
In this section, we apply the finite difference method to \eqref{she} to get a spatially semi-discrete scheme, whose continuous version solution can be written into a compact integral form by the use of explicit expression of the semi-discrete Green function. The spatially semi-discrete scheme is convergent to the exact solution in the mean square sense with order $\frac{1}{2}$. Based on the detailed analysis on the semi-discrete Green function and reverse Gr{\"o}nwall's inequality, we prove that the numerical solution of this semi-discretization is weakly intermittent. Moreover, this semi-discrete scheme preserves the sharp exponential order of the second moment of the exact solution.

\subsection{Spatially semi-discrete scheme}\label{sec3.1}
We introduce the uniform partition on the spatial domain $[0,1]$ with step size $\frac{1}{n}$ for a fixed integer $n\ge 3$.
Let $u^n(t,\frac{k}{n})$ be the approximation of $u(t,\frac{k}{n})$, $k=0,1,\ldots,n-1$.
The spatially semi-discrete scheme based on the finite difference method is given by:
\begin{equation}\label{fdm}
\begin{aligned}
\begin{cases}
du^n\left(t,\frac{k}{n}\right)=n^2\left(u^n\left(t,\frac{k+1}{n}\right)-2u^n\left(t,\frac{k}{n}\right)+u^n\left(t,\frac{k-1}{n}\right)\right)dt+\lambda \sqrt{n}\sigma\left(u^n\left(t,\frac{k}{n}\right)\right)dW^n_k(t),\vspace{1ex}\\
u^n(t,0)=u^n(t,1),\quad u^n\left(t,-\frac{1}{n}\right)=u^n\left(t,\frac{n-1}{n}\right),\quad t\ge 0,\vspace{1ex}\\
u^n\left(0,\frac{k}{n}\right)=u_0\left(\frac{k}{n}\right),\quad k=0,1,\ldots, n-1,
\end{cases}
\end{aligned}
\end{equation}
where $W^n_k(t):=\sqrt{n}\big(W\big(t,\frac{k+1}{n}\big)-W\big(t,\frac{k}{n}\big)\big).$
By the linear interpolation with respect to the space variable, it follows from Appendix \ref{A.2} that the mild form of $u^n$ is given by:
\begin{align}\label{mild fdm}
u^n(t,x)=\int_{0}^{1}G^n(t,x,y)u^n(0,\left(\kappa_n(y)\right)\,dy+\lambda\int_{0}^{t}\int_{0}^{1}G^n(t-s,x,y)\sigma\left(u^n(s,\kappa_n(y))\right)dW(s,y),
\end{align}
almost surely for all $t\ge 0$ and $x\in [0,1]$, where
$G^n(t,x,y):=\sum_{j=0}^{n-1}e^{\lambda^n_j t}e^n_j(x)\bar{e}_j(\kappa_n(y))$
 with $\lambda^n_j:=-4n^2\sin^2\left(\frac{j\pi}{n}\right)$, $\kappa_n(y):=\frac{[ny]}{n}$, $[\cdot]$ being the greatest integer function, $e_j(x)=e^{2\pi \mathbf{i}jx}$, $\bar{e}_j(\cdot)$ representing the complex conjugate of $e_j(\cdot)$ and
\begin{align*}
e^n_j(x):=e_j\left(\kappa_n(x)\right)+(nx-n\kappa_n(x))\Big[e_j\Big(\kappa_n(x)+\frac{1}{n}\Big)-e_j(\kappa_n(x))\Big],\quad \forall x\in[0,1].
\end{align*}

Nevertheless, based on the periodicity of $\lambda^n_j$ and $e_j$ with respect to $j$, $G^n(t,x,y)$ can be rewritten into two cases:\\
\begin{align*}
G^n(t,x,y)=\left\{\begin{array}{ll}
\sum_{j=-\left[\frac{n}{2}\right]}^{\left[\frac{n}{2}\right]}e^{\lambda^n_jt}e^n_j(x)\bar{e}_j(\kappa_n(y)),\quad &n\text{ is odd},\\
\sum_{j=-\frac{n}{2}+1}^{\frac{n}{2}}e^{\lambda^n_jt}e^n_j(x)\bar{e}_j(\kappa_n(y)),\quad &n \text{ is even}.
\end{array}
\right.
\end{align*}
By expanding the real and imaginary parts of $G^n$, it is not difficult to observe that $G^n$ is a real function (see Appendix \ref{aplemma3.3}). 
Now we give the main result of this subsection.
\begin{thm}\label{thm3.1}
Under Assumption \ref{Assumption 2}, the solution of the spatially semi-discrete scheme is weakly intermittent. 
\end{thm}
The proof of Theorem \ref{thm3.1} follows from Sections \ref{sec3.2} and \ref{sec3.3}. Before that, we prove the following properties of the semi-discrete Green function $G^n$, which is essential in establishing the weak intermittency of \eqref{fdm}.
\begin{lemma}\label{lemma 2.1}
$G^n(t,x,y)$ has the following properties:\vspace{1 ex}\\
(\romannumeral1) $\int_{0}^{1}G^n(t,x,y)\,dy=1$ for $t>0,\;x\in[0,1]$.\vspace{1 ex}\\
(\romannumeral2) 
For $t>0,\;x\in[0,1],$ the following equalities hold:
\begin{align*}
\int_{0}^{1}\left(G^n(t,x,y)\right)^2\,dy=\sum_{j=0}^{n-1}e^{2\lambda^n_jt}\left|e^n_j(x)\right|^2=
\left\{\begin{array}{ll}
\sum_{j=-\left[\frac{n}{2}\right]}^{\left[\frac{n}{2}\right]}e^{2\lambda^n_jt}\left|e^n_j(x)\right|^2,&n\text{ is odd},\vspace{1 ex}\\
\sum_{j=-\frac{n}{2}+1}^{\frac{n}{2}}e^{2\lambda^n_jt}\left|e^n_j(x)\right|^2,&n\text{ is even}.
\end{array}
\right.
\end{align*}
Moreover, $\int_{0}^{1}\left(G^n(t,x,y)\right)^2\,dy\ge 1$ for $t>0,\;x\in[0,1]$.\vspace{1 ex}\\
(\romannumeral3) $\int_{0}^{1}\left(G^n(t,x,y)\right)^2\,dy\leq 1+\sqrt{\frac{\pi}{8t}}$ for all $t>0,x\in[0,1]$.\vspace{1 ex}\\
(\romannumeral4) For each fixed $n\ge 3,$ there exists a number $t(n)>0$ depending on $n$, such that $G^n(t,x,y)\ge \frac{1}{2}>0$ for all $t>t(n),\;x,y\in[0,1]$.
\end{lemma}
\begin{proof}
$(\romannumeral1)$ For all $t\ge 0,x\in [0,1]$, we get
\begin{equation*}
\begin{aligned}
\int_{0}^{1}G^n(t,x,y)\,dy
&=\sum_{k=0}^{n-1}\frac{1}{n}\sum_{j=0}^{n-1}e^{\lambda^n_jt}e^n_j(x)e^{-2\pi\textbf{i}j\frac{k}{n}}\\
&=1+\sum_{j=1}^{n-1}\frac{1}{n}e^{\lambda^n_jt}e^n_j(x)\sum_{k=0}^{n-1}\cos\left(2\pi j\frac{k}{n}\right)=1,
\end{aligned}
\end{equation*}
where we have used the fact that $\sum_{k=0}^{n-1}\cos\left(2\pi j\frac{k}{n}\right)=0$ for $j\notin n\mathbb{Z}.$

%Similarly, if $n$ is even, we can also get the desired result.\\
$(\romannumeral2)$ For all $t\ge 0,x\in [0,1]$, taking advantage of the orthogonality of $\left\{e_j\right\}_{j=0,1,\ldots,n-1},$ we get
\begin{equation*}
\begin{aligned}
&\int_{0}^{1}\left(G^n(t,x,y)\right)^2\,dy\\
&=\sum_{k=0}^{n-1}\frac{1}{n}\sum_{j=0}^{n-1}e^{2\lambda^n_jt}\left|e^n_j(x)\right|^2\left|\bar{e}_j\left(\frac{k}{n}\right)\right|^2+\frac{1}{n}\sum_{j\neq l}e^{(\lambda^n_j+\lambda^n_l)t}e^n_j(x)\bar{e}^n_l(x)\sum_{k=0}^{n-1}\bar{e}_j\left(\frac{k}{n}\right)e_l\left(\frac{k}{n}\right)\\
&=\sum_{j=0}^{n-1}e^{2\lambda^n_jt}\left|e^n_j(x)\right|^2.
\end{aligned}
\end{equation*}
Similarly, we can get the result in the case of $n$ being odd and even.

$(\romannumeral3)$ By $(\romannumeral2)$, we have
\begin{align*}
\int_{0}^{1}\left(G^n(t,x,y)\right)^2\,dy\leq&\; 1+4\sum_{j=1}^{\left[\frac{n}{2}\right]}e^{2\lambda^n_jt}=1+4\sum_{j=1}^{\left[\frac{n}{2}\right]}e^{-8j^2\pi^2c^n_jt}
\leq 1+4\sum_{j=1}^{\left[\frac{n}{2}\right]}e^{-32j^2t}\\
\leq&\; 1+4\int_0^{\left[\frac{n}{2}\right]}e^{-32z^2t}\,dz\leq 1+4\int_0^{\infty}e^{-32z^2t}\,dz
\leq 1+\sqrt{\frac{\pi}{8t}},
\end{align*}
where we have used the fact that $c^n_j:=\sin^2(\frac{j\pi}{n})/(\frac{j\pi}{n})^2\in \left[\frac{4}{\pi^2},1\right]$ for $j=1,2,\ldots,\left[\frac{n}{2}\right]$. 

$(\romannumeral4)$ 
Since 
\begin{align}\label{semi2}
4n^2\sin ^2\left(\frac{j\pi}{n}\right)t\ge 4n^2\sin ^2\left(\frac{\pi}{n}\right)t,\quad j=1,2,\ldots,n-1, 
\end{align}
 and the right hand side of \eqref{semi2} converges to infinity as $t\rightarrow \infty$, we get
 \begin{align*}
 \left|\sum_{j=1}^{n-1}e^{\lambda^n_jt}e^{2\pi \textbf{i}j\frac{q-l}{n}}\right|\leq \sum_{j=1}^{n-1}e^{\lambda^n_1t}\rightarrow 0 \quad \text{as}\quad t\rightarrow \infty
 \end{align*}
 uniformly for all $q,l=0,1,\ldots,n-1.$ Hence, for each fixed $n,$ there is a positive constant $t(n)$ such that when $t>t(n),$
$
 -\frac{1}{2}\leq \sum_{j=1}^{n-1}e^{\lambda^n_jt}e^{2\pi\textbf{i} j\frac{q-l}{n}}\leq \frac{1}{2}
$
holds for all $q,l=0,1,\ldots,n-1.$ Therefore, when $t>t(n)$ and for all $q,l=0,1,\ldots,n-1,$ we have
\begin{align*}
G^n\left(t,\frac{q}{n},\frac{l}{n}\right)=1+\sum_{j=1}^{n-1}e^{\lambda^n_jt}e^{2\pi\textbf{i} j\frac{q-l}{n}}\ge \frac{1}{2}.
\end{align*}
This will lead to our desired result after linear interpolation with respect to the space variable.\\
Hence the proof is completed.
\end{proof}

\subsection{Intermittent upper bound}\label{sec3.2}
To give the a priori estimation of the mild solution to \eqref{mild fdm}, we
introduce norms on the space of random fields,
$$
\mathcal{N}_{\beta,p}(u):=\sup_{t\ge 0}\sup_{x\in[0,1]}\left\{e^{-\beta t}\left\|u(t,x)\right\|_p\right\},\quad \forall \beta>0,\;p\ge 2,
$$
where $\|\cdot\|_p$ denotes the $L^p(\Omega)$-norm. Let $\mathcal{L}^{\beta, p}$ be the completion of simple random fields in $\mathcal{N}_{\beta,p}$-norm. For more details, we refer to \cite[Chapter 4]{KD14}.
\begin{prop}\label{proposition 2.2}
%Let Assumption \ref{Assumption 2} hold, then 
There exists a random field $u^n\in \bigcup_{\beta >0}\mathcal{L}^{\beta,p}$ solving \eqref{mild fdm} for each $n\ge 3,\; p\ge 2$. Moreover, $u^n$ is a.s.-unique among all random fields satisfying
\begin{align*}
\sup_{x\in[0,1]}\mathbb{E}\left(\left|u^n(t,x)\right|^p\right)\leq C_1^p\exp\left\{C_2L_{\sigma}^4\lambda^4p^3t\right\},\quad \text{for }p\ge 2,\; t\ge 0,
\end{align*}
with some constants $C_1:=C_1\big(\sup_{x\in[0,1]}u_0(x),n\big)>0$ and $C_2>0.$
\end{prop}
\begin{proof}
We apply Picard's iteration by 
defining
\begin{align*}
&u^n_{(0)}(t,x):=u(0,x),\\
&u^n_{(q+1)}(t,x):=\int_{0}^{1}G^n(t,x,y)u(0,\kappa_n(y))\,dy+\lambda\int_{0}^{t}\int_{0}^{1}G^n(t-s,x,y)\sigma\left(u^n_{(q)}(s,\kappa_n(y))\right)\,dW(s,y).
\end{align*}
Using Lemma \ref{lemma 2.1} $(\romannumeral2)$ $(\romannumeral3)$, combining the linear growth of $\sigma$, Minkowski inequality and Burkholder-Davis-Gundy inequality, we obtain
\begin{align}\label{semi3}
\left\|u^n_{(q+1)}(t,x)\right\|^2_p\leq& \;2\sup_{x\in[0,1]}\left|u_0(x)\right|^2\times \int_0^1\left(G^n(t,x,y)\right)^2\,dy\nonumber\\
&+8p\lambda^2\int_0^t\int_0^1\left(G^n(t-s,x,y)\right)^2\left\|\sigma\left(u^n_{(q)}(s,\kappa_n(y))\right)\right\|^2_k\,ds\,dy\nonumber\\
\leq &\;2\sup_{x\in[0,1]}\left|u_0(x)\right|^2\times \left(2\sum_{j=0}^{n-1}
e^{2\lambda^n_jt}\right)\nonumber\\
&+CL^2_{\sigma}p\lambda^2\left(\sqrt{t}+t\right)+CL^2_{\sigma}p\lambda^2\int_{0}^{t}\left(\frac{1}{\sqrt{t-s}}+1\right)\sup_{y\in[0,1]}\left\|u^n_{(k)}(s,y)\right\|^2_p\,ds.
\end{align}
Multiplying $e^{-2\beta t}$ with $2\beta\ge1$ on both sides of \eqref{semi3}, taking supremum over $x\in[0,1],t\ge 0$, and noticing $2\sum_{j=0}^{n-1}e^{2\lambda^n_jt}\leq 2n$, we get
\begin{align*}
\mathcal{N}^2_{\beta,p}\left(u^n_{(q+1)}\right)
%&\sup_{t\ge 0}\sup_{x\in[0,1]}\left\{e^{-\beta t}\left\|u^n_{(k+1)}(t,x)\right\|^2_p\right\}\\
\leq &\;4\sup_{x\in[0,1]}\left|u_0(x)\right|^2\times n+\frac{CL^2_{\sigma}p\lambda^2}{\sqrt{4\beta e}}+\frac{CL^2_{\sigma}p\lambda^2}{2\beta e}+CL^2_{\sigma}p\lambda^2\left(\sqrt{\frac{\pi}{2\beta}}+\frac{1}{2\beta}\right)\mathcal{N}^2_{\beta,p}\left(u^n_{(q)}\right)\\
%\sup_{s\ge 0}\sup_{y\in[0,1]}\left\{e^{-\beta s}\left\|u^n_{(k)}(s,y)\right\|^2_p\right\}.
%\end{align*}
%Because $2\beta\ge1,$ we have $\sqrt{\frac{\pi}{2\beta}}+\frac{1}{2\beta}\leq \frac{3}{\sqrt{2\beta}},$ and hence
%\begin{align*}
%\mathcal{N}^2_{\beta,p}\left(u^n_{(k+1)}\right)
\leq &\;4\sup_{x\in[0,1]}\left|u_0(x)\right|^2\times n+\frac{3CL^2_{\sigma}p\lambda^2}{\sqrt{2\beta}}+\frac{3CL^2_{\sigma}p\lambda^2}{\sqrt{2\beta}}\mathcal{N}^2_{\beta,p}\left(u^n_{(q)}\right),
\end{align*}
where in the last step we have used $\sqrt{\frac{\pi}{2\beta}}+\frac{1}{2\beta}\leq \frac{3}{\sqrt{2\beta}}$ for $\beta\ge\frac{1}{2}.$\\
%Let $\beta :=18C^2K^4p^2\lambda^4$, and $C$ is chosen to be large enough to ensure $2\beta\ge 1$, then
There exists a $\beta$ such that 
\begin{align}\label{beta}
\frac{3CL^2_{\sigma}p\lambda^2}{2\beta}\leq\frac{1}{2} \quad\text{and}\quad \beta\ge \frac{1}{2}.
\end{align}
For example, one can choose $\beta=18C^2L^2_{\sigma}p^2\lambda^4+\frac{1}{2}.$ For such $\beta,$ 
we have
\begin{align}\label{semi5}
\mathcal{N}^2_{\beta,p}\left(u^n_{(q+1)}\right)
\leq&\;4\sup_{x\in[0,1]}\left|u_0(x)\right|^2\times n+\frac{1}{2}+\frac{1}{2}\mathcal{N}^2_{\beta,p}\left(u^n_{(q)}\right)\nonumber\\
 =&:\eta +\frac{1}{2}\mathcal{N}^2_{\beta,p}\left(u^n_{(q)}\right)
\leq 2\eta +\sup_{x\in[0,1]}\left|u_0(x)\right|^2=:C_1, 
\end{align}
which yields $u^n_{(q+1)}\in\mathcal{L}^{\beta,p}$.\\
Eq. \eqref{semi5} implies that for all $t\ge 0, x\in[0,1]$ and $\beta$ satisfying \eqref{beta},
\begin{align*}
\mathbb{E}\left(\left|u^n_{(q+1)}(t,x)\right|^p\right)\leq C^{\frac{p}{2}}_1\exp\left\{\beta t\right\}
\end{align*}
for each $p\ge 2,q\ge 0$.

Similarly, using the technique as before, we can prove
\begin{align*}
\mathcal{N}^2_{\beta,p}\left(u^n_{(q+1)}-u^n_{(q)}\right)\leq \frac{3CL^2_{\sigma}p\lambda^2}{\sqrt{2\beta}}\mathcal{N}^2_{\beta,p}\left(u^n_{(q)}-u^n_{(q-1)}\right).
\end{align*}
By choosing
$\beta$ satisfying \eqref{beta},
we obtain that $\left\{u^n_{(q)}(t,x)\right\}_{q\ge 0}$ is a Cauchy sequence in $\mathcal{N}_{\beta,p}$-norm, i.e. $\left\{u^n_{(q)}(t,x)\right\}_{q\ge 0}$ converges to some random field $u^n$ in $\mathcal{N}_{\beta,p}$-norm for each fixed $p\ge 2$. Since $\mathcal{L}^{\beta,p}$ is complete, we deduce that $u^n\in\mathcal{L}^{\beta,p}$. Moreover, $u^n$ satisfies the integral equation \eqref{mild fdm} in $\mathcal{N}_{\beta,p}$-norm.

The uniqueness of the numerical solution in $\mathcal{N}_{\beta,p}$-norm can be shown in a similar way as above. Thus the proof is completed.
\end{proof}

Based on Proposition \ref{proposition 2.2}, we can give the upper bound of the upper $p$th moment Lyapunov exponent of numerical solution of the spatially semi-discrete scheme.
\begin{prop}\label{thm3.3}
There exists a positive constant $C$, such that for each $n\ge 3,\;p\in[2,\infty)$, we have
\begin{align*}
\sup_{x\in[0,1]}\bar{\gamma}^n_p(x):=\sup_{x\in[0,1]}\limsup_{t\rightarrow \infty}\frac{1}{t}\log \mathbb{E}\left(\left|u^n(t,x)\right|^p\right)\leq CL^4_{\sigma}\lambda^4p^3.
\end{align*}
%for all $x\in[0,1],\;p\in[2,\infty)$. 
%Therefore, for each $n\ge 3$, we get $\bar{\gamma}^n_p(x)<\infty$ for all $p\in[2,\infty),\;x\in[0,1]$.
\end{prop} 

\subsection{Intermittent lower bound}\label{sec3.3}
It remains to investigate the lower bound for the upper $2$th moment Lyapunov exponent.
Before that, we give the following reverse Gr{\"o}nwall's inequality.
\begin{lemma}{(Reverse Gr{\"o}nwall's inequality)}\label{lemma 2.4}
Let $\phi$ be nonnegative and satisfy $$ \phi (t)\ge \alpha+\beta \int_{0}^{t}\phi(s)\,ds$$ for $t>a>0,$ where $\alpha,\beta>0$ are constants, then for $t>a$,
$$
\phi(t)\ge  e^{\beta (t-a)}\left(\alpha+\beta\int_0^a\phi(s)\,ds\right).
$$
\end{lemma}
\begin{proof}
Note that $\phi$ satisfies $ \phi (t)\ge \left(\alpha+\beta\int_0^a\phi(s)\,ds\right)+\beta \int_{a}^{t}\phi(s)\,ds$ for $t>a$, we can easily get the desired result.
\end{proof}
\begin{prop}\label{thm3.5}
Under Assumption \ref{Assumption 2}, we have
\begin{align*}
\inf_{x\in[0,1]}\bar{\gamma}^n_2(x)\ge \lambda^2J^2_0>0.
\end{align*}
\end{prop}
\begin{proof}
For each fixed $n\ge 3,$ taking the second moment on both sides of \eqref{mild fdm}, combining Walsh isometry and Lemma \ref{lemma 2.1} $(\romannumeral1)$ $(\romannumeral2)$ $(\romannumeral4)$, we get when $t>t(n),$
\begin{align*}
&\mathbb{E}\left(\left|u^n(t,x)\right|^2\right)\\
=&\;\Big|\int_{0}^{1}G^n(t,x,y)u^n(0,\kappa_n(y))\,dy\Big|^2+\lambda^2\int_{0}^{t}\int_{0}^{1}\left(G^n(t-s,x,y)\right)^2\mathbb{E}\left(\left|\sigma(u^n(s,\kappa_n(y)))\right|^2\right)\,dsdy\\
\ge&\; I^2_0\Big|\int_{0}^{1}G^n(t,x,y)\,dy\Big|^2+\lambda^2J^2_0\int_{0}^{t}\int_{0}^{1}\left(G^n(t-s,x,y)\right)^2\mathbb{E}\left(\left|u^n(s,\kappa_n(y))\right|^2\right)\,dsdy\\
\ge&\; I^2_0+\lambda^2J^2_0\int_{0}^{t}\int_{0}^{1}\left(G^n(t-s,x,y)\right)^2\,dy\inf_{y\in[0,1]}\mathbb{E}\left(\left|u^n(s,y)\right|^2\right)\,ds\\
\ge&\; I^2_0+\lambda^2J^2_0\int_{0}^{t}\inf_{y\in[0,1]}\mathbb{E}\left(\left|u^n(s,y)\right|^2\right)\,ds.
\end{align*}
Taking infimum over $x\in[0,1]$, we have
\begin{align*}
\inf_{x\in[0,1]}\mathbb{E}\left(\left|u^n(t,x)\right|^2\right)\ge I^2_0+\lambda^2J^2_0\int_{0}^{t}\inf_{y\in[0,1]}\mathbb{E}\left(\left|u^n(s,y)\right|^2\right)\,ds.
\end{align*}
Applying Lemma \ref{lemma 2.4} with $\alpha=I^2_0,\beta=\lambda^2J^2_0, a=t(n)$, we obtain 
\begin{align*}
\inf_{x\in[0,1]}\mathbb{E}\left(\left|u^n(t,x)\right|^2\right)\ge I^2_0e^{\lambda^2J^2_0(t-t(n))},\quad \text{for } t>t(n),
\end{align*}
which leads to $$\inf_{x\in[0,1]}\bar{\gamma}^n_2(x)=\inf_{x\in[0,1]}\limsup_{t\rightarrow \infty}\frac{1}{t}\log \mathbb{E}\left(\left|u^n(t,x)\right|^2\right)\ge \lambda^2J^2_0>0.$$
Hence we finish the proof.
\end{proof}
\subsection{Sharp exponential order of the second moment}
It is shown in Section \ref{sec2} that the second moment of the solution of \eqref{she} has the sharp exponential order $C\lambda^4t$ under Assumption \ref{Assumption 2}. By applying a renewal approach, we can get the same kind of result for the numerical solution of the semi-discrete scheme for large $n$, provided additionally that the initial data is a positive constant.
\begin{assumption}\label{assumption3}
We assume that $u_0:\equiv I_0>0,$ and the spatial partition number $n$ satisfies $n\ge\zeta\lambda^2$ with some constant $\zeta>0.$
\end{assumption}
\begin{thm}\label{thm5.1}
Under Assumptions \ref{Assumption 2} and \ref{assumption3}, for each $n,$ we have
%there is constants $\mu:=\mu(n)\in(0,1)$ such that 
\begin{align}\label{sec5eq6}
\inf_{x\in[0,1]}\mathbb{E}\left(\left|u^n(t,\kappa_n(x))\right|^2\right)\ge Ce^{C_2 J^4_0\lambda^4t},\quad t>T,
\end{align}
where $T:=T(n)>0,\;C_1=\frac{8\pi\zeta I^2_0}{J^2_0+8\pi\zeta},\;C_2=\frac{2\zeta^2\pi^2 }{(J^2_0+8\pi\zeta)^2}$.
\end{thm}
Before giving the proof of Theorem \ref{thm5.1}, we present the refined property of the semi-discrete Green function and a probability density function for the renewal approach.
\begin{lemma}\label{lemma 5.1}
For $t>0,x\in[0,1],$ we have
\begin{align*}
\int_0^1\left(G^n(t,\kappa_n(x),y)\right)^2\,dy\ge\frac{1-e^{-2n^2\pi^2t}}{\sqrt{32\pi t}}.
\end{align*}
\end{lemma}
\begin{proof}
 %Let's assume $n$ is odd, we can get the same result in the case of $n$ being even.\\
Since $\left|e^n_j(\kappa_n(x))\right|^2=1,$ we obtain
\begin{align*}
&\int_0^1\left(G^n(t,\kappa_n(x),y)\right)^2\,dy=\sum_{j=0}^{n-1}e^{2\lambda^n_jt}\ge \sum_{j=0}^{\left[\frac{n}{2}\right]}e^{-8j^2\pi^2t}\ge \int_0^{\frac{n}{2}}e^{-8z^2\pi^2t}\,dz\\
=&\sqrt{\int_0^{\frac{n}{2}}\int_0^{\frac{n}{2}}e^{-8(z^2+w^2)\pi^2t}\,dw\,dz}=\sqrt{\frac{1}{4}\int_{-\frac{n}{2}}^{\frac{n}{2}}\int_{-\frac{n}{2}}^{\frac{n}{2}}e^{-8(z^2+w^2)\pi^2t}\,dw\,dz}\\
\ge& \sqrt{\frac{1}{4}\int_0^{2\pi}\int_0^{\frac{n}{2}}e^{-8r^2\pi^2t}r\,dr\,d\theta}=\frac{1}{\sqrt{32\pi}}\sqrt{\frac{1-e^{-2n^2\pi^2t}}{t}}\ge\frac{1-e^{-2n^2\pi^2t}}{\sqrt{32\pi t}},
\end{align*}
where we have used the polar coordinate transformation in the last line. The proof is finished.
\end{proof}

\begin{lemma}\label{lemmag}
Let $b:=\frac{\lambda^2J^2_0}{\sqrt{32\pi}}$ and $n\ge \zeta\lambda^2$. Then $g(t)=be^{-\pi\mu^2b^2t}\times \frac{1-e^{-2n^2\pi^2t}}{\sqrt{t}}$ is a probability density function on $[0,\infty)$ with some suitable $\mu\ge \frac{8\pi\zeta}{J^2_0+8\pi\zeta}>0$. 
%Moreover, $\mu\ge \frac{8\pi\zeta}{J^2_0+8\pi\zeta}.$
\end{lemma}
\begin{proof}
It suffices to find some $\mu>0$ such that
\begin{align*}
\int_0^{\infty}be^{-\pi\mu^2b^2t}\times \frac{1-e^{-2n^2\pi^2t}}{\sqrt{t}}\,dt=1,
\end{align*}
or equivalently, to prove that the continuous function
\begin{align*}
h(\mu):=\frac{b}{\sqrt{\mu^2b^2+2n^2\pi}}-\left(\frac{1}{\mu}-1\right)
\end{align*}
has a zero point $\mu>0$.
%Denote $h(\mu):=\frac{1}{\sqrt{\mu^2b^2+2n^2\pi^2}}-\left(\frac{1}{\mu}-1\right)\frac{1}{b}.$ 
Since $n\ge\zeta\lambda^2,$ so $h(\mu)\leq \sqrt{\frac{b^2}{\mu^2b^2+2\zeta^2\lambda^4\pi}}-\left(\frac{1}{\mu}-1\right)\leq \sqrt{\frac{J^4_0}{64\zeta^2\pi^2}}-\left(\frac{1}{\mu}-1\right)$, which implies $h(0^+)<0$. It is obvious that $h(1^-)>0$ for each fixed $n.$ Hence, there exists a $\mu\in(0,1)$ such that $h(\mu)=0,$ and $g(t)$ is a probability density function with this $\mu$. Moreover, $\mu=1/\big(\sqrt{\frac{b^2}{2n^2\pi+\mu^2b^2}}+1\big)\ge \frac{8\pi\zeta}{J^2_0+8\pi\zeta}.$ The proof is finished.
\end{proof}

\noindent
\textbf{Proof of Theorem \ref{thm5.1}:}
\begin{proof}
%Let's assume $n$ is odd, we can get the same result in the case of $n$ being even. 
 Taking the second moment on both sides of \eqref{mild fdm} with the space variable being $\kappa_n(x)$, combining Walsh isometry, Lemma \ref{lemma 2.1} $(\romannumeral1)$ $(\romannumeral2)$ and Lemma \ref{lemma 5.1}, we get
\begin{align}\label{sharp1}
\mathbb{E}\left(\left|u^n(t,\kappa_n(x))\right|^2\right)&\ge I^2_0+\lambda^2J^2_0\int_0^t\int_0^1\left(G^n(t-s,\kappa_n(x),y)\right)^2\,dy\inf_{y\in[0,1]}\mathbb{E}\left(\left|u^n(s,\kappa_n(y))\right|^2\right)\,ds\nonumber\\
&\ge I^2_0+\frac{\lambda^2J^2_0}{\sqrt{32\pi}}\int_0^t\frac{1-e^{-2n^2\pi^2(t-s)}}{\sqrt{t-s}}\inf_{y\in[0,1]}\mathbb{E}\left(\left|u^n(s,\kappa_n(y))\right|^2\right)\,ds.
\end{align}
Taking infimum over $x\in[0,1]$, then multiplying $e^{-\pi\mu^2b^2t}$ on both sides of \eqref{sharp1} with $b:=\frac{\lambda^2J^2_0}{\sqrt{32\pi}}$ and $\mu$ being a parameter that will be determined later, and denoting $$M^n(t):=e^{-\pi\mu^2b^2t}\inf_{x\in[0,1]}\mathbb{E}\left(\left|u^n(t,\kappa_n(x))\right|^2\right),$$ we obtain
\begin{align*}
M^n(t)\ge e^{-\pi\mu^2b^2t}I^2_0+\int_0^tbe^{-\pi\mu^2b^2(t-s)}\times \frac{1-e^{-2n^2\pi^2(t-s)}}{\sqrt{t-s}}M^n(s)\,ds.
\end{align*}
Consider 
\begin{align}\label{sec5eq4}
e^{-\pi\mu^2b^2t}f(t)= e^{-\pi\mu^2b^2t}I^2_0+\int_0^tg(t-s)e^{-\pi\mu^2b^2s}f(s)\,ds,
\end{align}
where $g(t)$ is defined as in Lemma \ref{lemmag} and is a probability density function. Hence, Renewal Theorem (see \cite[Theorem 8.5.14]{AKLS06}) ensures
\begin{align*}
\lim_{t\rightarrow\infty}e^{-\pi\mu^2b^2t}f(t)=\frac{\int_0^{\infty}e^{-\pi\mu^2b^2t}I^2_0\,dt}{\int_0^{\infty}tg(t)\,dt}\ge\frac{\int_0^{\infty}e^{-\pi\mu^2b^2t}I^2_0\,dt}{\int_0^{\infty}b\sqrt{t}e^{-\pi\mu^2b^2t}\,dt}=2\mu I^2_0.
\end{align*}
Therefore, there exists $T:=T(n)>0$, such that
\begin{align}\label{sec5eq5}
f(t)\ge \mu I^2_0e^{\pi\mu^2b^2t}, \quad \forall t>T.
\end{align}
Observing that $M^n(t)$ is a super-solution to \eqref{sec5eq4} and applying \cite[Theorem 7.11]{KD14}, we have
\begin{align*}
M^n(t)\ge e^{-\pi\mu^2b^2t}f(t),\quad \forall t>0,
\end{align*}
which together with \eqref{sec5eq5} implies 
\begin{align*}
\inf_{x\in[0,1]}\mathbb{E}\left(\left|u^n(t,\kappa_n(x))\right|^2\right)\ge \mu I^2_0e^{\pi\mu^2b^2t},\quad \forall t>T.
\end{align*}
Moreover,  by Lemma \ref{lemmag}, we have $\mu^2b^2\ge \frac{2\zeta^2\pi J^4_0\lambda^4}{(J^2_0+8\pi\zeta)^2}.$
This leads to \eqref{sec5eq6}.
%In order to apply the Renewal Theory, we need the function $g(t)$ to be a probability density function on $[0,\infty)$. To this end, 
Hence we complete the proof of the theorem.
\end{proof}
Theorem \ref{thm5.1} and Proposition \ref{proposition 2.2} indicate that the second moment of numerical solution to the spatially semi-discrete scheme grows at most and at least as $\exp\left\{C\lambda^4t\right\}$ as $t\rightarrow\infty.$

\subsection{Error estimations of semi-discrete scheme}\label{sec3.4}
In this subsection, we present the convergence result of the spatially semi-discrete scheme.
It is based on the error estimates of $G^n(t,x,y)$ and $G(t,x,y)$, whose proofs are postponed to Appendix \ref{aplemma3.3}. 
\begin{lemma}\label{lemma3.3}
(\romannumeral1) There exists a constant $C>0$ such that
\begin{align*}
\int_{0}^{\infty}\int_{0}^{1}\left|G(t,x,y)-G^n(t,x,y)\right|^2\,dy\,dt\leq \frac{C}{n}
\end{align*}
for all $x\in [0,1]$ and $n\ge 3$.\vspace{1 ex}\\
(\romannumeral2) For any $\frac{1}{2}<\alpha<1$, there exists a constant $C:=C(\alpha)>0$ such that 
\begin{align*}
\int_{0}^{1}\left|G(t,x,y)-G^n(t,x,y)\right|^2\,dy\leq Cn^{1-2\alpha}t^{-\alpha}
\end{align*}
for all $x\in[0,1],\;t>0$ and $n\ge 3$.
\end{lemma}
Based on Lemma \ref{lemma3.3}, we can establish the convergence theorem of the spatially semi-discrete scheme. Here, we omit its proof since it can be proved in a similar way as in \cite{Gy98}.
\begin{thm}
For every $0<\alpha<\frac{1}{2},\;p\ge 1$ and for every $t>0$, there is a constant $C:=C(\alpha,p,t)>0$ such that  
\begin{align*}
\sup_{x\in[0,1]}\left\|u^n(t,x)-u(t,x)\right\|_{2p}\leq Cn^{-\alpha}.
\end{align*}
\end{thm}

\section{Intrinsic property-preserving full discretization}\label{sec4}
In this section, we discretize \eqref{fdm} in temporal direction by the $\theta$-scheme to get a fully discrete scheme, whose solution can be written into a compact integral form by finding explicit expressions of the fully discrete Green functions. The fully discrete scheme is convergent to the exact solution in the mean square sense with order $\frac{1}{2}$ in the spatial direction and order $\frac{1}{4}$ in the temporal direction. Based on the technical estimates of the fully discrete Green functions, the numerical solution of this full discretization is proven to be weakly intermittent and to preserve the sharp exponential order of the second moment of the exact solution.
\subsection{Fully discrete scheme}\label{sec4.1}
We fix the uniform time step size $0<\tau<1$. In the sequel, we always assume $n\ge 3$. By using the $\theta$-scheme to discretize \eqref{fdm}, we obtain the following fully discrete scheme:
%If we use  finite difference method for spacial discretization and $\theta$-scheme for time discretization, we obtain the following scheme
\begin{align}\label{fullscheme}
\begin{cases}
u^{n,\tau}(t_{i+1},x_j)=u^{n,\tau}(t_i,x_j)+(1-\theta)\tau \Delta_nu^{n,\tau}(t_i,\cdot)(x_j)+\theta \tau \Delta_nu^{n,\tau}(t_{i+1},\cdot)(x_j) \vspace{1 ex}\\
\qquad \qquad \qquad \quad+\lambda\tau \sigma\left(u^{n,\tau}(t_i,x_j)\right)\square_{n,\tau}W(t_i,x_j), \vspace{1 ex}\\
u^{n,\tau}(t_i,0)=u^{n,\tau}(t_i,1),\quad u^{n,\tau}\left(t_i,\frac{1}{n}\right)=u^{n,\tau}\left(t_i,\frac{n-1}{n}\right),\quad i=0,1,\ldots, \vspace{1 ex}\\
u^{n,\tau}(0,x_j)=u_0(x_j),\quad j=0,1,\ldots, n-1,
\end{cases}
\end{align} 
where $u^{n,\tau}$ is an approximation of $u^n$, $ t_i:=i\tau,x_j:=\frac{j}{n}$, and
\begin{equation*}
\begin{aligned}
&\Delta_nu^{n,\tau}(t_i,\cdot)(x_j):=n^2\left(u^{n,\tau}(t_i,x_{j+1})-2u^{n,\tau}(t_i,x_j)+u^{n,\tau}(t_i,x_{j-1})\right),\vspace{1 ex}\\
&\square_{n,\tau}W(t_i,x_j):=n\tau^{-1}\left(W(t_{i+1},x_{j+1})-W(t_i,x_{j+1})-W(t_{i+1},x_j)+W(t_i,x_j)\right).
\end{aligned}
\end{equation*}
%functions $\varphi$ and $\psi$ are defined on $\{x_j:j=0,1,\ldots, n\}$ and on the lattice 
%$$
%\mathcal{L}:=\{(t_i,x_j):i=0,1,\ldots, m,\quad j=0,1,\ldots, n\},
%$$
%respectively. 
By the linear interpolation with respect to the space variable, i.e., for $i=0,1,\ldots,$
%We extend $u^{n,\tau}$ from $\mathcal{L}$ to $[0,T]\times[0,1]$ by polygonal interpolation. The extension by polygonal interpolation of a function $f$ given on $\mathcal{L}$ is defined as follows
\begin{equation*}
\begin{aligned}
u^{n,\tau}(t_i,x):=u^{n,\tau}(t_i,\kappa_n(x))+n(x-\kappa_n(x))\left[u^{n,\tau}\Big(t_i,\kappa_n(x)+\frac{1}{n}\Big)-u^{n,\tau}(t_i,\kappa_n(x))\right],
\end{aligned}
\end{equation*}
the mild form of $u^{n,\tau}$ is given by:
\begin{align}\label{mild full}
u^{n,\tau}(t,x)=&\;\int_{0}^{1}G^{n,\tau}_1(t,x,y)u_0(\kappa_n(y))\,dy\nonumber\\
&+\lambda\int_{0}^{t}\int_{0}^{1}G^{n,\tau}_2(t-\kappa_{\tau}(s)-\tau,x,y)\sigma \left(u^{n,\tau}(\kappa_{\tau}(s),\kappa_n(y))\right)\,dW(s,y),
\end{align}
almost surely for every $t=i\tau, x\in [0,1]$, where the fully discrete Green functions
\begin{align*}
&G^{n,\tau}_1(t,x,y):=\sum_{l=0}^{n-1}\left(R_{1,l}R_{2,l}\right)^{\left[\frac{t}{\tau}\right]}e^n_l(x)\bar{e}_l(\kappa_n(y)),\\
&G^{n,\tau}_2(t,x,y):=\sum_{l=0}^{n-1}\left(R_{1,l}R_{2,l}\right)^{\left[\frac{t}{\tau}\right]}R_{1,l}e^n_l(x)\bar{e}_l(\kappa_n(y))
\end{align*}
with $R_{1,l}:=(1-\theta \tau\lambda^n_l)^{-1}, R_{2,l}:=1+(1-\theta)\tau \lambda^n_l$, $\kappa_{\tau}(s):=\left[\frac{s}{\tau}\right]\tau$. 
For the derivation of \eqref{mild full}, we refer to Appendix \ref{sec5.4}.

Moreover, $G^{n,\tau}_i,i=1,2$ can be rewritten as
\begin{align*}
G^{n,\tau}_1(t,x,y)=
\left\{\begin{array}{ll}
\sum_{l=-\left[\frac{n}{2}\right]}^{\left[\frac{n}{2}\right]}\left(R_{1,l}R_{2,l}\right)^{\left[\frac{t}{\tau}\right]}e^n_l(x)\bar{e}_l(\kappa_n(y)),\quad& n\text{ is odd},\\
\sum_{l=-\frac{n}{2}+1}^{\frac{n}{2}}\left(R_{1,l}R_{2,l}\right)^{\left[\frac{t}{\tau}\right]}e^n_l(x)\bar{e}_l(\kappa_n(y)),\quad &n\text{ is even,}
\end{array}
\right.
\end{align*}
\begin{align*}
G^{n,\tau}_2(t,x,y)=
\left\{\begin{array}{ll}
\sum_{l=-\left[\frac{n}{2}\right]}^{\left[\frac{n}{2}\right]}\left(R_{1,l}R_{2,l}\right)^{\left[\frac{t}{\tau}\right]}R_{1,l}e^n_l(x)\bar{e}_l(\kappa_n(y)),\quad& n\text{ is odd},\\
\sum_{l=-\frac{n}{2}+1}^{\frac{n}{2}}\left(R_{1,l}R_{2,l}\right)^{\left[\frac{t}{\tau}\right]}R_{1,l}e^n_l(x)\bar{e}_l(\kappa_n(y)),\quad &n\text{ is even.}
\end{array}
\right.
\end{align*}
By expanding the real and imaginary parts, it is not difficult to observe that $G^{n,\tau}_i,i=1,2$ are real functions (see Appendix \ref{aplemma4.8} $(\romannumeral1)$).

Below, we give the definitions of the upper $p$th moment Lyapunov exponent (see \cite{ FD12, talay91}) and weak intermittency for the fully discrete numerical solution.
\begin{dfn}
(\romannumeral1) For the numerical solution $u^{n,\tau}$ of a fully discrete scheme, its upper $p$th moment Lyapunov exponent at $x\in[0,1]$ is defined by
\begin{align}
\bar{\gamma}^{n,\tau}_p(x):=\limsup_{m\rightarrow \infty}\frac{1}{m\tau}\log\mathbb{E}\left(\left|u^{n,\tau}(m\tau,x)\right|^p\right),
\end{align}
for $p\in(0,\infty)$.\vspace{1 ex}\\
(\romannumeral2) The numerical solution $u^{n,\tau}$ is called weakly intermittent if for all $x\in[0,1]$, $\bar{\gamma}^{n,\tau}_2(x)>0$ and $\bar{\gamma}^{n,\tau}_p(x)<\infty$ for $p> 2$.
\end{dfn}

Before we investigate the weak intermittency of the fully discrete scheme, we first present some conditions on step sizes to ensure the well-posedness of the fully discrete Green functions. That is to say, the step sizes are chosen to such that $|R_{1,j}R_{2,j}|<1,$ $j=1,2,\ldots,\left[\frac{n}{2}\right].$ Note that $R_{1,j}R_{2,j}<1,$ so what we need is to find conditions such that
\begin{align*}
R_{1,j}R_{2,j}=\frac{1+(1-\theta)\tau\lambda^n_j}{1-\theta\tau\lambda^n_j}\ge-1+\epsilon, \quad j=1,2,\ldots,\left[\frac{n}{2}\right]
\end{align*}
 for some fixed $\epsilon>0.$
It is equivalent to
\begin{align}\label{sec4eq2}
-4(1-2\theta+\epsilon\theta)n^2\tau\sin^2\frac{j\pi}{n}\ge-2+\epsilon.
\end{align}
Hence, we divide $\theta$ into the following three cases.

\textit{Case 1: $\theta \in[0,\frac{1}{2})$}. For such $\theta,$ we have $1-2\theta+\epsilon\theta>0,$ hence, \eqref{sec4eq2} is equivalent to
\begin{align}\label{sec4eq1}
n^2\tau\sin^2\frac{j\pi}{n}\leq\frac{2-\epsilon}{4(1-2\theta+\epsilon\theta)}.
\end{align}
Suppose $n^2\tau\leq r\leq \frac{2-\epsilon}{4(1-2\theta+\epsilon\theta)},$ then $\epsilon\leq 2-\frac{4r}{1+4\theta r},$ and \eqref{sec4eq1} holds for $j=1,2,\ldots,\left[\frac{n}{2}\right].$ Moreover, $2-\frac{4r}{1+4\theta r}>0$ implies $r<\frac{1}{2-4\theta}.$

\textit{Case 2: $\theta =\frac{1}{2}$}. We suppose $n^2\tau\leq \frac{2-\epsilon}{4(1-2\theta+\epsilon\theta)}=\frac{1}{\epsilon}-\frac{1}{2},$ then \eqref{sec4eq2} holds with $\theta=\frac{1}{2}$.

\textit{Case 3: $\theta \in (\frac{1}{2},1]$}. For such $\theta,$ we can choose $\epsilon>0$ small enough, e,g., $\epsilon:=\min\left\{-\frac{1-2\theta}{2\theta},\frac{1}{2}\right\},$  such that \eqref{sec4eq2} holds for all $n\ge 3,\;0<\tau<1,$ $j=1,2,\ldots,\left[\frac{n}{2}\right].$

%In order to get the weak intermittency of the fully discrete scheme, 

To sum up, we make the following assumptions on the spatial step size $\frac{1}{n}$ and the temporal step size $\tau$ when $\theta$ takes different values.
\begin{assumption}\label{Assumption3}
(\romannumeral1) For $0\leq \theta<\frac{1}{2},$ suppose 
$n^2\tau\leq r<\frac{1}{2-4\theta}$ with some constant $r>0$.\\
(\romannumeral2) For $\theta=\frac{1}{2},$ suppose 
$n^2\tau\leq \frac{1}{\epsilon}-\frac{1}{2}$ with some $\epsilon\in(0,\frac{1}{2}).$\\
(\romannumeral3) For $\frac{1}{2}<\theta\leq 1$, there is no coupled requirement for $n,\tau.$
\end{assumption}
Below, we give the main result of this subsection.
\begin{thm}\label{thm4.2}
Under Assumptions \ref{Assumption 2} and \ref{Assumption3}, the solution of the fully discrete scheme is weakly intermittent.
\end{thm}
The proof of Theorem \ref{thm4.2} follows from the intermittent upper bound (Sections \ref{sec4.2}) and the intermittent lower bound (Section \ref{sec4.3}). Before that, we prove some properties of the fully discrete Green functions, which play a key role in the estimates of the intermittent upper and lower bounds.
In the following, we define $R_{3,j}:=(R_{1,j}R_{2,j})^{-1}-1=-\frac{\lambda^n_j\tau}{1+(1-\theta)\tau\lambda^n_j}.$
\begin{lemma}\label{lemma4.1}
For $n\ge 3,\;0<\tau<1,$ $G^{n,\tau}_i(t,x,y),i=1,2$ have the following properties:\vspace{1 ex}\\
(\romannumeral1) $\int_{0}^{1}G^{n,\tau}_1(t,x,y)\,dy=1$ for $t>0,\;x\in[0,1]$.\vspace{1 ex}\\
(\romannumeral2) For $t>0,\;x\in[0,1],$ the following equalities hold:
\begin{align*}
\int_{0}^{1}\left(G^{n,\tau}_2(t,x,y)\right)^2\,dy&=\sum_{j=0}^{n-1}(R_{1,j}R_{2,j})^{2\left[\frac{t}{\tau}\right]}R^2_{1,j}\big|e^n_j(x)\big|^2\\
&=\left\{\begin{array}{ll}
\sum_{j=-\left[\frac{n}{2}\right]}^{\left[\frac{n}{2}\right]}(R_{1,j}R_{2,j})^{2\left[\frac{t}{\tau}\right]}R^2_{1,j}\big|e^n_j(x)\big|^2,&n\text{ is odd},\vspace{1 ex}\\
\sum_{j=-\frac{n}{2}+1}^{\frac{n}{2}}(R_{1,j}R_{2,j})^{2\left[\frac{t}{\tau}\right]}R^2_{1,j}\big|e^n_j(x)\big|^2,& n\text{ is even.}
\end{array}
\right.
\end{align*}
Moreover, we have
$\int_{0}^{1}\left(G^{n,\tau}_2(t,x,y)\right)^2\,dy\ge 1$.\vspace{1 ex}\\
(\romannumeral3) Under Assumption \ref{Assumption3}, $\int_{0}^{1}\left(G^{n,\tau}_2(t,x,y)\right)^2\,dy\leq 1+\frac{C}{\sqrt{\left[\frac{t}{\tau}\right]\tau+\tau}}$ with some constant $C:=C(\theta)>0$ for all $t>0,x\in[0,1]$.\vspace{1 ex}\\
(\romannumeral4) Under Assumption \ref{Assumption3}, for each fixed $n\ge 3$ and $0<\tau<1,$ there exists a number $t(n,\tau)>0$ depending on $n,\tau,$ such that $G^{n,\tau}_1(t,x,y)\ge \frac{1}{2}>0$ for all $t>t(n,\tau),\;x,y\in[0,1]$.
\end{lemma}
\begin{proof}
The proofs of $(\romannumeral1)$ $(\romannumeral2)$ are similar to those in Lemma \ref{lemma 2.1}, so we only prove $(\romannumeral3)$ $(\romannumeral4)$.

$(\romannumeral3)$
\iffalse
In order to ensure $|R_{1,j}R_{2,j}|<1.$ For fixed $\epsilon>0,$ the inequality
\begin{align*}
R_{1,j}R_{2,j}=\frac{1+(1-\theta)\tau\lambda^n_j}{1-\theta\tau\lambda^n_j}\ge-1+\epsilon
\end{align*}
is equivalent to
\begin{align}\label{sec4eq2}
-4(1-2\theta+\epsilon\theta)n^2\tau\sin^2\frac{j\pi}{n}\ge-2+\epsilon.
\end{align}
Hence, we divide $\theta$ into the following three cases.\\
\textit{Case 1: $\theta \in[0,\frac{1}{2})$}. For such $\theta,$ we have $1-2\theta+\epsilon\theta>0,$ hence, \eqref{sec4eq2} is equivalent to
\begin{align}\label{sec4eq1}
n^2\tau\sin^2\frac{j\pi}{n}\leq\frac{2-\epsilon}{4(1-2\theta+\epsilon\theta)}.
\end{align}
Suppose $n^2\tau\leq r\leq \frac{2-\epsilon}{4(1-2\theta+\epsilon\theta)},$ then $\epsilon\leq 2-\frac{4r}{1+4\theta r},$ and \eqref{sec4eq1} holds for $j=1,2,\ldots,\left[\frac{n}{2}\right].$ Moreover, $2-\frac{4r}{1+4\theta r}>0$ implies $r<\frac{1}{2-4\theta}.$\\
\textit{Case 2: $\theta =\frac{1}{2}$}. We suppose $n^2\tau\leq \frac{2-\epsilon}{4(1-2\theta+\epsilon\theta)}=\frac{1}{\epsilon}-\frac{1}{2},$ then \eqref{sec4eq2} holds with $\theta=\frac{1}{2}$.\\
\textit{Case 3: $\theta \in (\frac{1}{2},1]$}. For such $\theta,$ we can choose $\epsilon>0$ small enough, e,g., $\epsilon:=\min\left\{-\frac{1-2\theta}{2\theta},\frac{1}{2}\right\},$  such that \eqref{sec4eq2} holds for all $n\ge 3,\;0<\tau<1,$ $j=1,2,\ldots,\left[\frac{n}{2}\right].$
\fi
We split the set $\left\{j:1,2,\ldots,\left[\frac{n}{2}\right]\right\}$ into two parts, i.e.,
\begin{align*}
\left\{j:1,2,\ldots,\left[\frac{n}{2}\right]\right\}&=\left\{j:R_{1,j}R_{2,j}\ge \frac{1}{2}\right\}\cup\left\{j:-1+\epsilon\leq R_{1,j}R_{2,j}<\frac{1}{2}\right\}\\
&=:A_1\cup A_2.
\end{align*}

In the sequel, we always use the fact that for $j\in A_1,$ $\frac{1}{2}<R_{2,j}<1$ and $-\lambda^n_j\tau\leq R_{3,j}\leq -2\lambda^n_j\tau$, and for $j\in A_2$, $\left|R_{1,j}R_{2,j}\right|\leq 1-\epsilon$. Moreover, we observe that $A_1\subset \left\{j:1\leq j\leq \frac{1}{4}\sqrt{\frac{1}{(2-\theta)\tau}}\right\}$ and $A_2\subset \left\{j:\frac{1}{2\pi}\sqrt{\frac{1}{(2-\theta)\tau}}<j\leq \left[\frac{n}{2}\right]\right\}.$

Hence,
\begin{align*}
&\int_{0}^{1}\left(G^{n,\tau}_2(t,x,y)\right)^2\,dy
%\leq &1+2\sum_{j=1}^{\left[\frac{n}{2}\right]}\left(R_{1,j}R_{2,j}\right)^{2\left[\frac{t}{\tau}\right]}R^2_{1,j}\left(\left(\varphi^n_{c,j}(x)\right)^2+\left(\varphi^n_{s,j}(x)\right)^2\right)\\
\leq 1+4\sum_{j=1}^{\left[\frac{n}{2}\right]}\left(R_{1,j}R_{2,j}\right)^{2\left[\frac{t}{\tau}\right]}R^2_{1,j}\\
=&\;1+4\sum_{j\in A_1}\left(R_{1,j}R_{2,j}\right)^{2\left[\frac{t}{\tau}\right]}R^2_{1,j}+4\sum_{j\in A_2}\left(R_{1,j}R_{2,j}\right)^{2\left[\frac{t}{\tau}\right]}R^2_{1,j}\\
%=&\;1+4\sum_{j\in A_1}\left(1+R_{3,j}\right)^{-2\left[\frac{t}{\tau}\right]-2}R_{2,j}^{-2}+4\sum_{j\in A_2}\left(R_{1,j}R_{2,j}\right)^{2\left[\frac{t}{\tau}\right]}R^2_{1,j}\\
\leq&\;1+16\sum_{j\in A_1}\left(1+R_{3,j}\right)^{-2\left[\frac{t}{\tau}\right]-2}+4\sum_{j\in A_2}(1-\epsilon)^{2\left[\frac{t}{\tau}\right]}R^2_{1,j}
=:1+J_1+J_2.
\end{align*}

We split $J_1$ further as follows,
\begin{align*}
J_1=&16\sum_{j\in A_1}\left(1+R_{3,j}\right)^{-2\left[\frac{t}{\tau}\right]-2}\\
=&\;16\sum_{j\in A_1}\left(\left(1+R_{3,j}\right)^{-2\left[\frac{t}{\tau}\right]-2}-\exp\left\{-2R_{3,j}\left(\left[\frac{t}{\tau}\right]+1\right)\right\}\right)\\
&+16\sum_{j\in A_1}\exp\left\{-2R_{3,j}\left(\left[\frac{t}{\tau}\right]+1\right)\right\}\\
=&:J_{1,1}+J_{1,2}.
\end{align*}
For the term $J_{1,2}$, 
\begin{align*}
J_{1,2}&\leq 16\sum_{j\in A_1}e^{-32j^2\tau \left(\left[\frac{t}{\tau}\right]+1\right)}\leq16\sum_{1\leq j\leq\frac{1}{4}\sqrt{\frac{1}{(2-\theta)\tau}}}e^{-32j^2\tau \left(\left[\frac{t}{\tau}\right]+1\right)}\\
&\leq 16\int_0^{\infty}e^{-32z^2\tau\left(\left[\frac{t}{\tau}\right]+1\right)}\,dz\leq C\left(\left[\frac{t}{\tau}\right]\tau+\tau\right)^{-\frac{1}{2}}.
\end{align*}
As for $J_{11},$ 
%noticing $-\lambda^n_j\tau<R_{3,j}<-2\lambda^n_j\tau$ for $j\in A_1$ and $A_1\subset \left\{j:1\leq j\leq \frac{1}{4}\sqrt{\frac{1}{(2-\theta)\tau}}\right\},$ we have
\begin{align*}
J_{1,1}
&\leq 16\sum_{j\in A_1}\exp\left\{-2\left(\left[\frac{t}{\tau}\right]+1\right)\ln \left(1+R_{3,j}\right)\right\}\\
&\qquad \qquad \times\left(1-\exp\left\{2\left(\left[\frac{t}{\tau}\right]+1\right)\left(-R_{3,j}+\ln\left(1+R_{3,j}\right)\right)\right\}\right)\\
&\leq 16\sum_{j\in A_1}\exp\left\{-2\left(\left[\frac{t}{\tau}\right]+1\right)\ln \left(1-\lambda^n_j\tau\right)\right\}\\
&\qquad \qquad \times \left(1-\exp\left\{2\left(\left[\frac{t}{\tau}\right]+1\right)\left(2\lambda^n_j\tau+\ln\left(1-2\lambda^n_j\tau\right)\right)\right\}\right)\\
&\leq 16\sum_{1\leq j\leq\frac{1}{4}\sqrt{\frac{1}{(2-\theta)\tau}}}\exp\left\{-2\left(\left[\frac{t}{\tau}\right]+1\right)\ln \left(1-\lambda^n_j\tau\right)\right\}\\
& \qquad\qquad\qquad\qquad\times\left(1-\exp\left\{2\left(\left[\frac{t}{\tau}\right]+1\right)\left(2\lambda^n_j\tau+\ln\left(1-2\lambda^n_j\tau\right)\right)\right\}\right)\\
&\leq 16\sum_{1\leq j\leq\frac{1}{4}\sqrt{\frac{1}{(2-\theta)\tau}}}\exp\left\{2\left(\left[\frac{t}{\tau}\right]+1\right)C_2\lambda^n_j\tau\right\}\times \left(2\left(\left[\frac{t}{\tau}\right]+1\right)C_1\left(2\lambda^n_j\tau\right)^2\right)\\
&\leq\sum_{1\leq j\leq\frac{1}{4}\sqrt{\frac{1}{(2-\theta)\tau}}}C\left|\left(\left[\frac{t}{\tau}\right]+1\right)j^2\tau \right|^{-\frac{3}{2}}\left(\left[\frac{t}{\tau}\right]+1\right)j^4\tau^2\\
&\leq \sum_{1\leq j\leq\frac{1}{4}\sqrt{\frac{1}{(2-\theta)\tau}}}Cj\left(\left[\frac{t}{\tau}\right]+1\right)^{-\frac{1}{2}}\tau^{\frac{1}{2}}\leq C(\theta)\left(\left[\frac{t}{\tau}\right]\tau+\tau\right)^{-\frac{1}{2}},
\end{align*}
where we have used the fact that $\lambda^n_j\tau \in\left[-4\pi^2j^2\tau,-16j^2\tau \right]$ for $j=1,2,\ldots,\left[\frac{n}{2}\right]$ and $j^2\tau<\frac{1}{16(2-\theta)}$, so $z:=-\lambda^n_j\tau\in \left(0,\frac{\pi^2}{4(2-\theta)}\right]$, for such $z$, we have $-C_1z^2\leq-z+\ln (1+z)\leq 0$ and $\ln (1+z)\ge C_2z$ for some $C_1,C_2>0$. The inequalities $-z+\ln(1+z)\leq 0$ with $z\ge 0,\;1-e^{-z}\leq z$ with $z\ge 0$ and $e^{-z^2}\leq C(\alpha)z^{-\alpha}$ with $\alpha>0, z>0$ are also used, here we choose $\alpha=3$.\\
%Because of $A_2\subset \left\{j:\frac{1}{2\pi}\sqrt{\frac{1}{(2-\theta)\tau}}<j\leq \left[\frac{n}{2}\right]\right\},$ we have

For the term $J_2,$
\begin{align*}
J_2&\leq \sum_{\frac{1}{2\pi}\sqrt{\frac{1}{(2-\theta)\tau}}<j\leq \left[\frac{n}{2}\right]}4(1-\epsilon)^{2\left[\frac{t}{\tau}\right]}(1-\theta \tau\lambda^n_j)^{-2}\leq \sum_{\frac{1}{2\pi}\sqrt{\frac{1}{(2-\theta)\tau}}<j\leq \left[\frac{n}{2}\right]}4(1-\epsilon)^{2\left[\frac{t}{\tau}\right]}(1+16\theta j^2\tau)^{-2}\\
&\leq 4\int_{\frac{1}{2\pi}\sqrt{\frac{1}{(2-\theta)\tau}}}^{\left[\frac{n}{2}\right]}(1-\epsilon)^{2\left[\frac{t}{\tau}\right]}(1+16\theta x^2\tau)^{-2}\,dx\quad (\text{let} \quad y=x\sqrt{\tau})\\
&\leq \frac{4}{\sqrt{\tau}}\int_{\frac{1}{2\pi}\sqrt{\frac{1}{2-\theta}}}^{\left[\frac{n}{2}\right]\sqrt{\tau}}(1-\epsilon)^{2\left[\frac{t}{\tau}\right]}(1+16\theta y^2)^{-2}\,dy\\
&\leq \frac{4}{\sqrt{\tau}}\times e^{-2\left(\left[\frac{t}{\tau}\right]+1\right)\ln (1-\epsilon)^{-1}}\times (1-\epsilon)^{-2}\int_{\frac{1}{2\pi}\sqrt{\frac{1}{2-\theta}}}^{\left[\frac{n}{2}\right]\sqrt{\tau}}(1+16\theta y^2)^{-2}\,dy\\
&\leq C(\theta)\left(\left[\frac{t}{\tau}\right]\tau+\tau\right)^{-\frac{1}{2}}\int_{\frac{1}{2\pi}\sqrt{\frac{1}{2-\theta}}}^{\left[\frac{n}{2}\right]\sqrt{\tau}}(1+16\theta y^2)^{-2}\,dy,
\end{align*}
where in the last line we use the inequality $e^{-z^2}\leq C(\alpha)z^{-\alpha}$, for $\alpha >0,	z>0$, and $\alpha$ is chosen to be 1. Therefore, it remains to prove $\int_{\frac{1}{2\pi}\sqrt{\frac{1}{2-\theta}}}^{\left[\frac{n}{2}\right]\sqrt{\tau}}(1+16\theta y^2)^{-2}\,dy\leq C$ for some $C>0$.\\
For the \textit{Case 1} and \textit{Case 2}, because $n^2\tau$ is bounded, so $\int_{\frac{1}{2\pi}\sqrt{\frac{1}{2-\theta}}}^{\left[\frac{n}{2}\right]\sqrt{\tau}}(1+16\theta y^2)^{-2}\,dy\leq C.$ \\
For the \textit{Case 3}, $$\int_{\frac{1}{2\pi}\sqrt{\frac{1}{2-\theta}}}^{\left[\frac{n}{2}\right]\sqrt{\tau}}(1+16\theta y^2)^{-2}\,dy\leq\int_{0}^{\infty}(1+16\theta y^2)^{-2}\,dy\leq C.$$
Combining these three cases, we finish the proof of $(\romannumeral3)$.

$(\romannumeral4)$ We only prove the case of $n$ being odd since the proof is similar when $n$ is even. For each fixed $n\ge3,0<\tau<1,$ $R_{1,j}R_{2,j}$ is a decreasing sequence of $j$. Hence, under Assumption \ref{Assumption3}, for all $j=1,2,\ldots,\left[\frac{n}{2}\right],$ we have 
$
-1+\epsilon\leq R_{1,j}R_{2,j}\leq R_{1,1}R_{2,1}<1.
$
Therefore, for each $n,\tau,$ we can choose $\epsilon':=\min\left\{\epsilon,1-R_{1,1}R_{2,1}\right\}>0$, such that $\left|R_{1,j}R_{2,j}\right|\leq 1-\epsilon'$ for $j=1,2,\ldots,\left[\frac{n}{2}\right].$ Then
\begin{align*}
2\left|\sum_{j=1}^{\left[\frac{n}{2}\right]}\left(R_{1,j}R_{2,j}\right)^{\left[\frac{t}{\tau}\right]}e_j(\kappa_n(x))\bar{e}_j(\kappa_n(y))\right|
\leq 2\sum_{j=1}^{\left[\frac{n}{2}\right]}\left|R_{1,j}R_{2,j}\right|^{\left[\frac{t}{\tau}\right]}\leq 2\sum_{j=1}^{\left[\frac{n}{2}\right]}(1-\epsilon')^{\left[\frac{t}{\tau}\right]}\rightarrow 0 
\end{align*}
$\text{as }t\rightarrow \infty$ for all $x,y\in[0,1].$ So there exists a $t:=t(n,\tau)>0$ large enough, such that when $t>t(n,\tau),$
\begin{align*}
-\frac{1}{2}\leq 2\sum_{j=1}^{\left[\frac{n}{2}\right]}\left(R_{1,j}R_{2,j}\right)^{\left[\frac{t}{\tau}\right]}e_j(\kappa_n(x))\bar{e}_j(\kappa_n(y))\leq \frac{1}{2},
\end{align*}
which implies
\begin{align*}
G^{n,\tau}_1(t,\kappa_n(x),y)=1+2\sum_{j=1}^{\left[\frac{n}{2}\right]}\left(R_{1,j}R_{2,j}\right)^{\left[\frac{t}{\tau}\right]}e_j(\kappa_n(x))\bar{e}_j(\kappa_n(y))\ge \frac{1}{2}
\end{align*}
for all $x,y\in[0,1]$ and $t>t(n,\tau).$ This will lead to our desired result after linear interpolation with respect to the space variable.
\end{proof}
\subsection{Intermittent upper bound}\label{sec4.2}
\begin{prop}\label{prop 3.2}
Under Assumption \ref{Assumption3}, there exists a random field $u^{n,\tau}\in \bigcup_{\beta>0}\mathcal{L}^{\beta,p}$ solving \eqref{mild full} for each $n\ge 3, 0<\tau<1$, $p\ge 2$. Moreover, $u^{n,\tau}$ is a.s-unique among all random fields satisfying
$$\sup_{x\in[0,1]}\mathbb{E}\left(\left|u^{n,\tau}(t,x)\right|^p\right)\leq C^p_1\exp\left\{C_2L^4_{\sigma}\lambda^4p^3t\right\},\quad \text{for } p\ge 2,\;t=m\tau,\;m\ge 0$$
with $C_1:=C_1\big(\sup_{x\in[0,1]}u_0(x),n\big)>0$ and $C_2>0$.
\end{prop}
\begin{proof}
We apply Picard's iteration again by 
defining
\begin{align*}
&u^{n,\tau}_{(0)}(t,x):=u_0(x),\\
&u^{n,\tau}_{(q+1)}(t,x):=\int_{0}^{1}G^{n,\tau}_1(t,x,y)u_0(\kappa_n(y))\,dy\\
&\qquad\qquad\qquad+\lambda\int_{0}^{t}\int_{0}^{1}G^{n,\tau}_2(t-\kappa_{\tau}(s)-\tau,x,y)\sigma\left(u^{n,\tau}_{(q)}(\kappa_{\tau}(s),\kappa_n(y))\right)\,dW(s,y).
\end{align*}
Using Lemma \ref{lemma4.1} $(\romannumeral2)$ $(\romannumeral3)$, combining the linear growth of $\sigma$, Minkowski inequality and Burkholder-Davis-Gundy inequality, we obtain
\begin{align*}
&\left\|u^{n,\tau}_{(q+1)}(m\tau,x)\right\|^2_p\\
%\leq& 2\sup_{x\in[0,1]}\left|u_0(x)\right|^2\times \int_0^1\left(G^{n,\tau}_1(m\tau,x,y)\right)^2\,dy\\&+8p\lambda^2\int_0^{m\tau}\int_0^1\left(G^{n,\tau}_{2}(m\tau-s,x,y)\right)^2\left\|\sigma\left(u^{n,\tau}_{(k)}(\kappa_{tau}(s),\kappa_n(y))\right)^2\right\|^2_k\,ds\,dy\\
\leq &\;2\sup_{x\in[0,1]}\left|u_0(x)\right|^2\times\Big(1+4\sum_{j=1}^{\left[\frac{n}{2}\right]}\left(R_{1,j}R_{2,j}\right)^{2m}\Big)\\
&+CL^2_{\sigma}p\lambda^2\int_0^{m\tau}\left(\frac{1}{\sqrt{m\tau-\kappa_{\tau}(s)}}+1\right)\left(1+\sup_{y\in[0,1]}\left\|u^{n,\tau}_{(q)}(\kappa_{\tau}(s),y)\right\|^2_p\right)\,ds\\
\leq &\;2\sup_{x\in[0,1]}\left|u_0(x)\right|^2\times\Big(1+4\sum_{j=1}^{\left[\frac{n}{2}\right]}\left(R_{1,j}R_{2,j}\right)^{2m}\Big)+CL^2_{\sigma}p\lambda^2\int_0^{m\tau}\left(\frac{1}{\sqrt{m\tau-s}}+1\right)\,ds\\
&+CL^2_{\sigma}p\lambda^2\int_0^{m\tau}\left(\frac{1}{\sqrt{m\tau-\kappa_{\tau}(s)}}+1\right)\sup_{y\in[0,1]}\left\|u^{n,\tau}_{(q)}(\kappa_{\tau}(s),y)\right\|^2_p\,ds.
\end{align*}
Under Assumption \ref{Assumption3}, we have $\left|R_{1,j}R_{2,j}\right|<1$, so
$
1+4\sum_{j=1}^{\left[\frac{n}{2}\right]}\left(R_{1,j}R_{2,j}\right)^{2m}\leq 1+2n\leq3n.
$
Therefore, 
\begin{align}\label{full3}
\left\|u^{n,\tau}_{(q+1)}(m\tau,x)\right\|^2_p
%\leq &\;6n\sup_{x\in[0,1]}\left|u_0(x)\right|^2\nonumber\\&+CL^2_{\sigma}p\lambda^2\int_0^{m\tau}\left(\frac{1}{\sqrt{m\tau-s}}+1\right)\,ds+CL^2_{\sigma}p\lambda^2\int_0^{m\tau}\left(\frac{1}{\sqrt{m\tau-\kappa_{\tau}(s)}}+1\right)\sup_{y\in[0,1]}\left\|u^{n,\tau}_{(q)}(\kappa_{\tau}(s),y)\right\|^2_p\,ds\nonumber\\
\leq&\;6n\sup_{x\in[0,1]}\left|u_0(x)\right|^2+CL^2_{\sigma}p\lambda^2(\sqrt{m\tau}+m\tau)\nonumber\\
&+CL^2_{\sigma}p\lambda^2\int_0^{m\tau}\left(\frac{1}{\sqrt{m\tau-\kappa_{\tau}(s)}}+1\right)\sup_{y\in[0,1]}\left\|u^{n,\tau}_{(q)}(\kappa_{\tau}(s),y)\right\|^2_p\,ds.
\end{align}
Multiplying $e^{-2\beta m\tau}$ with $2\beta\ge 1$ on both sides of \eqref{full3} and taking supremum over $m\ge 0$, we obtain
\begin{align*}
&\sup_{m\ge 0}\sup_{x\in[0,1]}\left\{e^{-2\beta m\tau}\left\|u^{n,\tau}_{(q+1)}(m\tau,x)\right\|^2_p\right\}\\
\leq &\;6n\sup_{x\in[0,1]}\left|u_0(x)\right|^2+CL^2_{\sigma}p\lambda^2\left(\frac{1}{\sqrt{4\beta e}}+\frac{1}{2\beta e}\right)+CL^2_{\sigma}p\lambda^2\\
&\times\sup_{m\ge 0}\sum_{j=0}^{m-1}e^{-2\beta(m\tau-j\tau)}\int_{j\tau}^{(j+1)\tau}\left(\frac{1}{\sqrt{m\tau-j\tau}}+1\right)\,ds\sup_{j\ge 0}\sup_{y\in[0,1]}\left\{e^{-2\beta j\tau}\left\|u^{n,\tau}_{(q)}(j\tau,y)\right\|^2_p\right\}\\
\leq &\;6n\sup_{x\in[0,1]}\left|u_0(x)\right|^2+CL^2_{\sigma}p\lambda^2\left(\frac{1}{\sqrt{4\beta e}}+\frac{1}{2\beta e}\right)\\
&+CL^2_{\sigma}p\lambda^2\int_{0}^{\infty}e^{-2\beta r}\left(\frac{1}{\sqrt{r}}+1\right)\,dr\sup_{j\ge 0}\sup_{y\in[0,1]}\left\{e^{-2\beta j\tau}\left\|u^{n,\tau}_{(q)}(j\tau,y)\right\|^2_p\right\}\\
\leq &\;6n\sup_{x\in[0,1]}\left|u_0(x)\right|^2+CL^2_{\sigma}p\lambda^2\left(\frac{1}{\sqrt{4\beta e}}+\frac{1}{2\beta e}\right)\\&+CL^2_{\sigma}p\lambda^2\left(\sqrt{\frac{\pi}{2\beta}}+\frac{1}{2\beta}\right)\sup_{j\ge 0}\sup_{y\in[0,1]}\left\{e^{-2\beta j\tau}\left\|u^{n,\tau}_{(q)}(j\tau,y)\right\|^2_p\right\}\\
\leq &\;6n\sup_{x\in[0,1]}\left|u_0(x)\right|^2+\frac{3CL^2_{\sigma}p\lambda^2}{\sqrt{2\beta}}+\frac{3CL^2_{\sigma}p\lambda^2}{\sqrt{2\beta}}\sup_{j\ge 0}\sup_{y\in[0,1]}\left\{e^{-2\beta j\tau}\left\|u^{n,\tau}_{(q)}(j\tau,y)\right\|^2_p\right\},
\end{align*}
where in the last step we have used $\sqrt{\frac{\pi}{2\beta}}+\frac{1}{2\beta}\leq \frac{3}{\sqrt{2\beta}}$ for $2\beta\ge1.$

The remaining part of the proof is similar to that of Proposition \ref{proposition 2.2} by choosing $\beta=18C^2L^4_{\sigma}p^2\lambda^4+\frac{1}{2}$ that satisfies $\frac{3CL^2_{\sigma}p\lambda^2}{\sqrt{2\beta}}\leq\frac{1}{2}$ and $\beta\ge \frac{1}{2}$. Hence, we can get
\begin{align*}
\mathbb{E}\left(\big|u^{n,\tau}_{(q+1)}(m\tau,x)\big|^p\right)\leq C^{\frac{p}{2}}_1\exp\left\{\beta m\tau\right\},\quad p\ge 2,
\end{align*}
where $C_1=(12n+1)\sup_{x\in[0,1]}|u_0(x)|^2+1.$
 Moreover, by the similar technique as in Proposition \ref{proposition 2.2}, one can prove the convergence of $\left\{u^{n,\tau}_{(q)}\right\}_{q\ge 0}$ and the uniqueness of the solution of \eqref{mild full}.
We omit the details.
The proof is completed.
\end{proof}
Based on Proposition \ref{prop 3.2}, we give the following result, which shows the upper bound for the upper $p$th moment Lyapunov exponent.
\begin{prop}\label{thm4.4}
Under Assumption \ref{Assumption3}, there exists a positive constant $C$ such that for each fixed $0<\tau<1,\;n\ge 3$ and $p\in[2,\infty),$ we have
\begin{align*}
\sup_{x\in[0,1]}\bar{\gamma}^{n,\tau}_p(x)\leq CL^4_{\sigma}\lambda^4p^3.
\end{align*}
%Therefore, for each $0<\tau<1,n\ge 3$, we have $\bar{\gamma}^{n,\tau}_p<\infty$ for all $x\in[0,1],\;p\in[2,\infty).$
\end{prop}

\subsection{Intermittent lower bound}\label{sec4.3}
 It remains to investigate the lower bound of $\bar{\gamma}^{n,\tau}_2$.
Before that, we give the following reverse discrete Gr{\"o}nwall type inequality.
\begin{lemma}{(Reverse Discrete Gr{\"o}nwall type inequality)}\label{lemma 3.5}
Let $\{y_n\}_{n\ge 0}$ be nonnegative sequence and satisfy 
\begin{align}\label{discrete1}
y_n\ge \alpha+\sum_{0\leq k\leq n-1}\beta y_k
\end{align}
for $n\ge N$,where $\alpha,\beta>0$. Then for $l=0,1,2,\ldots,$
\begin{align}\label{discrete2}
y_{N+l}\ge \Big(\alpha+\beta\sum_{0\leq k\leq N-1}y_k\Big)(1+\beta)^{l}.
\end{align}
\end{lemma}
\begin{proof}
We prove \eqref{discrete2} by induction. When $l=0$, \eqref{discrete1} with $n=N$ implies $$y_N\ge \alpha+\sum_{0\leq k\leq N-1}\beta y_k,$$ which is \eqref{discrete2} with $l=0$.

Suppose that \eqref{discrete2} holds for all $l\leq q$, now we prove it in the case of $l=q+1$.
By \eqref{discrete1} with $n=N+q+1$ and the case of $l\leq q$, we get
\begin{align*}
y_{N+q+1}&\ge \alpha+\sum_{0\leq k\leq N-1}\beta y_k+\sum_{N\leq j\leq N+q}\beta y_j=\alpha+\sum_{0\leq k\leq N-1}\beta y_k+\sum_{0\leq j\leq q}\beta y_{N+j}\\
&=\alpha+\sum_{0\leq k\leq N-1}\beta y_k+\sum_{0\leq j\leq q}\beta \left(\alpha(1+\beta)^j+\beta (1+\beta)^j\sum_{0\leq k\leq N-1}y_k\right)\\
&=\alpha\left(1+\sum_{0\leq j\leq q}\beta (1+\beta)^j\right)+\beta\left(1+\sum_{0\leq j\leq q}\beta (1+\beta)^j\right)\sum_{0\leq k\leq N-1}y_k.
\end{align*}
It suffices to prove 
\begin{align}\label{equality}
1+\sum_{0\leq j\leq q}\beta(1+\beta)^j=(1+\beta)^{q+1},\quad q\ge 1.
\end{align}
To this end, we show it by induction again.

Obviously, \eqref{equality} holds for $q=1$. Suppose that it holds for $q=r-1$, we check it for $q=r$,
\begin{align*}
(1+\beta)^{r+1}&=\left(1+\sum_{0\leq j\leq r-1}\beta (1+\beta)^j\right)(1+\beta)\\
&=1+\sum_{0\leq j\leq r-1}\beta (1+\beta)^j+\beta\left(1+\sum_{0\leq j\leq r-1}\beta(1+\beta)^j\right)\\
&=1+\sum_{0\leq j\leq r-1}\beta (1+\beta)^j+\beta(1+\beta)^r=1+\sum_{0\leq j\leq r}\beta(1+\beta)^j.
\end{align*}
Hence we finish the proof.
\end{proof}
\begin{prop}\label{theorem 3.6}
Under Assumptions \ref{Assumption 2} and \ref{Assumption3}, for each fixed $n\ge 3$ and $0<\tau<1$, $$\inf_{x\in[0,1]}\bar{\gamma}^{n,\tau}_2(x)\ge \frac{\log (1+\lambda^2J^2_0\tau)}{\tau}>0.$$
%the numerical solution $u^{n,\tau}$ of the $\theta$-scheme is weakly intermittent when $n\ge 3,\;0<\tau<1$. Moreover, 
\end{prop}
\begin{proof}
For each fixed $n\ge 3,\;0<\tau<1,$ Lemma \ref{lemma4.1} $(\romannumeral4)$ implies that there is a $t(n,\tau)>0,$ such that $G^{n,\tau}_1(t,x,y)>0$ for $t>t(n,\tau).$ Hence, taking the second moment on both sides of \eqref{mild full}, combining Walsh isometry and  Lemma \ref{lemma4.1} $(\romannumeral1)$ $(\romannumeral2)$, we get when $m\tau>t(n,\tau),$
\begin{align*}
&\mathbb{E}\left(\left|u^{n,\tau}(m\tau,x)\right|^2\right)\\
\ge&\; I^2_0+\lambda^2J^2_0\int_{0}^{m\tau}\int_{0}^{1}G^{n,\tau}_2(m\tau-\kappa_{\tau}(s)-\tau,x,y)^2\mathbb{E}\left(\left|u^{n,\tau}(\kappa_{\tau}(s),\kappa_n(y))\right|^2\right)\,ds\,dy\\
\ge&\; I^2_0+\lambda^2J^2_0\int_{0}^{m\tau}\int_{0}^{1}G^{n,\tau}_2(m\tau-\kappa_{\tau}(s)-\tau,x,y)^2\,dy\inf_{y\in[0,1]}\mathbb{E}\left(\left|u^{n,\tau}(\kappa_{\tau}(s),y)\right|^2\right)\,ds\\
\ge&\; I^2_0+\lambda^2J^2_0\int_{0}^{m\tau}\inf_{y\in[0,1]}\mathbb{E}\left(\left|u^{n,\tau}(\kappa_{\tau}(s),y)\right|^2\right)\,ds.
\end{align*}
Taking infimum over $x\in[0,1]$ yields
\begin{align*}
\inf_{x\in[0,1]}\mathbb{E}\left(\left|u^{n,\tau}(m\tau,x)\right|^2\right)\ge I^2_0+\lambda^2J^2_0\int_{0}^{m\tau}\inf_{y\in[0,1]}\mathbb{E}\left(\left|u^{n,\tau}(\kappa_{\tau}(s),y)\right|^2\right)\,ds,
\end{align*}
which is equivalent to 
\begin{align*}
\inf_{x\in[0,1]}\mathbb{E}\left(\left|u^{n,\tau}(m\tau,x)\right|^2\right)\ge I^2_0+\lambda^2J^2_0\sum_{j=0}^{m-1}\inf_{y\in[0,1]}\mathbb{E}\left(\left|u^{n,\tau}(j\tau,y)\right|^2\right)\tau.
\end{align*}
Applying Lemma \ref{lemma 3.5} with $\alpha=I^2_0, \beta=\lambda^2J^2_0\tau, N=\left[\frac{t(n,\tau)}{\tau}\right]+1$ and omitting the last term on the right hand side of \eqref{discrete2}, we obtain
\begin{align*}
\inf_{x\in[0,1]}\mathbb{E}\left(\left|u^{n,\tau}(N\tau+l\tau,x)\right|^2\right)\ge I^2_0(1+\lambda^2J^2_0\tau)^l.
\end{align*}
This leads to 
\begin{align*}
\inf_{x\in[0,1]}\bar{\gamma}^{n,\tau}_2(x)=\inf_{x\in[0,1]}\limsup_{l\rightarrow \infty}\frac{\log \mathbb{E}\left(\left|u^{n,\tau}(N\tau+l\tau,x)\right|^2\right)}{N\tau+l\tau}\ge \frac{\log (1+\lambda^2J^2_0\tau)}{\tau}>0.
\end{align*}
The proof is finished.
\end{proof}
%Based on Theorem \ref{thm4.4} and Theorem \ref{theorem 3.6}, we have proved that under conditions in Lemma \ref{lemma4.1} (\romannumeral3) and Assumption \ref{Assumption 2}, $\bar{\gamma}^{n,\tau}_2(x)>0$ and $\bar{\gamma}^{n,\tau}_p(x)<\infty$ for all $x\in[0,1],\;p\in[2,\infty)$, i.e., the solution of \eqref{fullscheme} is weakly intermittent.
\begin{remark}
(\romannumeral1) By Theorem \ref{theorem 3.6}, we have
$$\inf_{x\in[0,1]}\liminf_{\tau\rightarrow 0}\bar{\gamma}^{n,\tau}_2(x)\ge \lim_{\tau\rightarrow 0}\frac{\log(1+\lambda^2J^2_0\tau)}{\lambda^2J^2_0\tau}\lambda^2J^2_0=\lambda^2J^2_0,$$
where this lower bound is equal to that of the spatial semi-discretization.\vspace{1 ex}\\
(\romannumeral2) As for the exponential integrator method (see \cite{Cohen20}), whose continuous version can be written into the following mild form:
\begin{align*}
u^{n,\tau}_E(t,x):=&\int_{0}^{1}G^n(t,x,y)u_0(\kappa_n(y))\,dy\\
&+\lambda\int_{0}^{t}\int_{0}^{1}G^n(t-\kappa_{\tau}(s),x,y)\sigma\left(u^{n,\tau}_E(\kappa_{\tau}(s),\kappa_n(y))\right)\,W(dsdy),
\end{align*}
for $t=i\tau, x\in[0,1]$, where 
$$G^n(t,x,y)=\sum_{j=0}^{n-1}e^{\lambda^n_j t}e^n_j(x)\bar{e}_j(\kappa_n(y)),$$
we can get the weak intermittency of this fully discrete scheme similarly.
\end{remark}

\subsection{Sharp exponential order of the second moment}
In this subsection, by applying a discrete renewal method, the second moment of the numerical solution of the fully discrete scheme is proved to have sharp exponential order $C\lambda^4t.$ 
%We show the main result firstly, and the proof will be postponed until after the Lemmas below.
\begin{thm}\label{thm4.9}
Let Assumptions \ref{Assumption 2} and \ref{assumption3} hold. Then for $n,\tau$ satisfying $\frac{J^4_0n^2\tau}{16\pi\zeta^2}+16\pi n^2\tau< 1$, we have
\begin{align}\label{sec5eq11}
\inf_{x\in[0,1]}\mathbb{E}\left(\left|u^{n,\tau}(m\tau,\kappa_n(x))\right|^2\right)\ge C_1e^{C_2J^4_0\lambda^4m\tau},\quad m\tau>T,
%\liminf_{\lambda\rightarrow\infty}\frac{\log\mathcal{E}^{n,\tau}_t(\lambda)}{\lambda^4}:=\liminf_{\lambda\rightarrow\infty}\frac{\log\sqrt{\mathbb{E}\left(\left\|u^{n,\tau}(t,\cdot)\right\|^2_{L^2[0,1]}\right)}}{\lambda^4}\ge \frac{4\pi^4\mu^2J^4_0t}{(4\pi+\theta)^4},\quad \forall t=m\tau>T.
\end{align}
%where $0<\epsilon\leq \frac{16\pi\zeta}{J^2_0+32\pi\zeta}$ is independent of $\lambda.$
where $T:=T(n,\tau)>0,\;C_1=\frac{16\pi\zeta I^2_0}{J^2_0+32\pi\zeta},\;C_2=\frac{4\pi^2\zeta^2}{(J^2_0+32\pi\zeta)^2}.$
\end{thm}
The proof of Theorem \ref{thm4.9} depends on the refined estimate of the fully discrete Green function and a discrete probability density function for the discrete renewal method.
\begin{lemma}\label{lemma4.10}
For $0\leq\theta<1$ with $n^2\tau<\frac{1}{8(1-\theta)}$ or $\theta=1,$ we have
\begin{align*}
\int_0^1\left(G^{n,\tau}_2(t,\kappa_n(x),y)\right)^2\,dy\ge \frac{1-\exp\left\{-4n^2\pi^2\left(\left[\frac{t}{\tau}\right]+1\right)\tau\right\}}{8\sqrt{\pi\left(\left[\frac{t}{\tau}\right]+1\right)\tau}}.
\end{align*}
\end{lemma}
\begin{proof}
Under conditions in Lemma \ref{lemma4.10}, we have $\frac{1}{2}<R_{2,j}\leq1,R_{3,j}<-2\lambda^n_j\tau$, hence,
\begin{align*}
&\int_0^1\left(G^{n,\tau}_2(t,\kappa_n(x),y)\right)^2\,dy
\ge\sum_{j=0}^{\left[\frac{n}{2}\right]}\left(1+R_{3,j}\right)^{-2\left[\frac{t}{\tau}\right]-2}(R_{2,j})^{-2}\\
\ge& \sum_{j=0}^{\left[\frac{n}{2}\right]}\exp\left\{-2\left(\left[\frac{t}{\tau}\right]+1\right)\ln\left(1+R_{3,j}\right)\right\}
\ge \sum_{j=0}^{\left[\frac{n}{2}\right]}\exp\left\{-2\left(\left[\frac{t}{\tau}\right]+1\right)R_{3,j}\right\}\\
\ge &\sum_{j=0}^{\left[\frac{n}{2}\right]}\exp\left\{4\tau\lambda^n_j\left(\left[\frac{t}{\tau}\right]+1\right)\right\}
\ge \sum_{j=0}^{\left[\frac{n}{2}\right]}\exp\left\{-16j^2\pi^2\left(\left[\frac{t}{\tau}\right]+1\right)\tau\right\}\\
\ge&\frac{1-\exp\left\{-4n^2\pi^2\left(\left[\frac{t}{\tau}\right]+1\right)\tau\right\}}{8\sqrt{\pi\left(\left[\frac{t}{\tau}\right]+1\right)\tau}},
\end{align*}
where in the last line we use the same technique as in Lemma \ref{lemma 5.1}.
The proof is finished.
\end{proof}

\begin{lemma}\label{lemmagr}
Let $\tilde{g}(r):=\tilde{b}e^{-\pi\mu^2\tilde{b}^2r\tau}\frac{1-e^{-4n^2\pi^2r\tau}}{\sqrt{r\tau}}\tau$, where $\tilde{b}:=\frac{\lambda^2J^2_0}{8\sqrt{\pi}}$. Suppose that $n,\tau$ satisfy $\frac{J^4_0n^2\tau}{16\pi\zeta^2}+16\pi n^2\tau< 1$ and $n\ge \zeta\lambda^2$ for some $\zeta>0$, then $\left\{\tilde{g}(r)\right\}_{r\ge 1}$ is a discrete probability density function with some suitable $\mu\ge \frac{16\pi\zeta}{J^2_0+32\pi\zeta}>0$.
\end{lemma}
\begin{proof}
 It suffices to find some constant $\mu>0$ to be the zero point of the function
\begin{align}\label{sec5eq14}
\tilde{h}(\mu):=\frac{1}{\tilde{b}}\Big(\sum_{r=1}^{\infty}\tilde{g}(r)-1\Big)= 
\sum_{r=1}^{\infty}\frac{e^{-\pi\mu^2\tilde{b}^2r\tau}}{\sqrt{r\tau}}\tau-\sum_{r=1}^{\infty}\frac{e^{-(\pi\mu^2\tilde{b}^2+4n^2\pi^2)r\tau}}{\sqrt{r\tau}}\tau-\frac{1}{\tilde{b}}.
\end{align}
On the one hand, 
\begin{align}\label{sec5eq8}
\sum_{r=1}^{\infty}\frac{e^{-\pi\mu^2\tilde{b}^2r\tau}}{\sqrt{r\tau}}\tau\leq \int_0^{\infty}\frac{e^{-\pi\mu^2\tilde{b}^2z\tau}}{\sqrt{z}}\sqrt{\tau}\,dz=\frac{1}{\mu \tilde{b}}.
\end{align}
On the other hand,
\begin{align}\label{sec5eq9}
&\sum_{r=1}^{\infty}\frac{e^{-\pi\mu^2\tilde{b}^2r\tau}}{\sqrt{r\tau}}\tau\ge \int_1^{\infty}\frac{e^{-\pi\mu^2\tilde{b}^2z\tau}}{\sqrt{z}}\sqrt{\tau}\,dz\\
=&\;\int_0^{\infty}\frac{e^{-\pi\mu^2\tilde{b}^2z\tau}}{\sqrt{z}}\sqrt{\tau}\,dz-\int_0^{1}\frac{e^{-\pi\mu^2\tilde{b}^2z\tau}}{\sqrt{z}}\sqrt{\tau}\,dz\ge \frac{1}{\mu \tilde{b}}-2\sqrt{\tau}.
\end{align}
Similarly, we have
\begin{align}\label{sec5eq10}
\sqrt{\frac{1}{\mu^2\tilde{b}^2+4n^2\pi}}-2\sqrt{\tau}\leq \sum_{r=1}^{\infty}\frac{e^{-(\pi\mu^2\tilde{b}^2+4n^2\pi^2)r\tau}}{\sqrt{r\tau}}\tau\leq \sqrt{\frac{1}{\mu^2\tilde{b}^2+4n^2\pi}}.
\end{align}
Combining \eqref{sec5eq8} \eqref{sec5eq9} and \eqref{sec5eq10}, we obtain
\begin{align*}
\frac{1}{\mu \tilde{b}}-2\sqrt{\tau}-\sqrt{\frac{1}{\mu^2\tilde{b}^2+4n^2\pi}}-\frac{1}{\tilde{b}}\leq h(\mu)\leq \frac{1}{\mu \tilde{b}}-\left(\sqrt{\frac{1}{\mu^2\tilde{b}^2+4n^2\pi}}-2\sqrt{\tau}\right)-\frac{1}{\tilde{b}}.
\end{align*}
Hence, for any $\varepsilon>0,$ the right hand side of \eqref{sec5eq14} converges uniformly to a continuous function on $\mu\in[\varepsilon,1]$, we still denote it by $\tilde{h}(\mu).$

 Because of $n\ge \zeta\lambda^2,$ by choosing $\varepsilon= \frac{1}{\sqrt{\frac{\tilde{b}^2}{4\pi\zeta^2\lambda^4}}+2}=\frac{16\pi\zeta}{J^2_0+32\pi\zeta}\leq\frac{1}{\sqrt{\frac{\tilde{b}^2}{4n^2\pi}}+2}$ that is independent of $\tilde{b},$ we have $$\left(\frac{1}{\varepsilon}-1\right)\ge\sqrt{\frac{\tilde{b}^2}{4n^2\pi}}+1>\sqrt{\frac{\tilde{b}^2}{4n^2\pi}}+2\tilde{b}\sqrt{\tau}>\sqrt{\frac{\tilde{b}^2}{\varepsilon^2\tilde{b}^2+4n^2\pi}}+2\tilde{b}\sqrt{\tau},$$ which yields $\tilde{h}(\varepsilon)>0$.
 
  Since $\frac{J^4_0n^2\tau}{16\pi\zeta^2}+16\pi n^2\tau< 1$ implies $1>4\tilde{b}^2\tau+16\pi n^2\tau$, so
$
\tilde{h}(1)\leq -\left(\sqrt{\frac{1}{\tilde{b}^2+4n^2\pi}}-2\sqrt{\tau}\right)< 0.
$
 Therefore, there is a $\mu:=\mu(n,\tau,\tilde{b})\in(\varepsilon,1)$ satisfying $\tilde{h}(\mu)=0$, and $\mu\ge\varepsilon=\frac{16\pi\zeta}{J^2_0+32\pi\zeta}$. The proof is finished.
\end{proof}
\noindent
\textbf{Proof of Theorem \ref{thm4.9}:}
\begin{proof}
Taking the second moment on both sides of \eqref{mild full} with space variable being $\kappa_n(x)$ and time variable being $m\tau$, combining Walsh isometry, Lemma \ref{lemma4.1} $(\romannumeral1)$ $(\romannumeral2)$ and Lemma \ref{lemma4.10}, we get
\begin{align*}
\mathbb{E}\left(\left|u^{n,\tau}(m\tau,\kappa_n(x))\right|^2\right)&\ge I^2_0+\frac{\lambda^2J^2_0}{8\sqrt{\pi}}\sum_{j=0}^{m-1}\frac{1-e^{-4n^2\pi^2(m-j)\tau}}{\sqrt{(m-j)\tau}}\inf_{y\in[0,1]}\mathbb{E}\left(\left|u^{n,\tau}(j\tau,\kappa_n(y))\right|^2\right)\tau.
\end{align*}
Taking infimum over $x\in[0,1],$ then multiplying both sides by $e^{-\pi\mu^2\tilde{b}^2m\tau}$ with $\tilde{b}=\frac{\lambda^2J^2_0}{8\sqrt{\pi}}$, we see that
\begin{align*}
M^{n,\tau}(m\tau):=e^{-\pi\mu^2\tilde{b}^2m\tau}\inf_{x\in[0,1]}\mathbb{E}\left(\left|u^{n,\tau}(m\tau,\kappa_n(x))\right|^2\right)
\end{align*}
satisfies
\begin{align*}
M^{n,\tau}(m\tau)\ge& e^{-\pi\mu^2\tilde{b}^2m\tau}I^2_0+\sum_{j=0}^{m-1}\tilde{b}e^{-\pi\mu^2\tilde{b}^2(m-j)\tau}\frac{1-e^{-4n^2\pi^2(m-j)\tau}}{\sqrt{(m-j)\tau}}M^{n,\tau}(j\tau)\tau\\
=&e^{-\pi\mu^2\tilde{b}^2m\tau}I^2_0+\sum_{j=1}^{m}\tilde{b}e^{-\pi\mu^2\tilde{b}^2j\tau}\frac{1-e^{-4n^2\pi^2j\tau}}{\sqrt{j\tau}}M^{n,\tau}((m-j)\tau)\tau.
\end{align*}
 By Lemma \ref{lemmagr}, $\tilde{g}(r)=\tilde{b}e^{-\pi\mu^2\tilde{b}^2r\tau}\frac{1-e^{-4n^2\pi^2r\tau}}{\sqrt{r\tau}}\tau$ is a discrete probability density function. Hence, applying the discrete Renewal Theorem (see \cite[Theorem 8.5.13]{AKLS06}) and the discrete version of \cite[Theorem 7.11]{KD14} lead to 
\begin{align*}
\liminf_{m\rightarrow\infty}M^{n,\tau}(m\tau)\ge \frac{I^2_0\sum_{r=0}^{\infty}e^{-\pi\mu^2\tilde{b}^2r\tau}}{\sum_{r=1}^{\infty}r \tilde{g}(r)}\ge \frac{I^2_0\int_0^{\infty}e^{-\pi\mu^2\tilde{b}^2z\tau}\,dz}{\int_0^{\infty}\tilde{b}e^{-\pi\mu^2\tilde{b}^2z\tau}\sqrt{z\tau}\,dz+1/(\mu\sqrt{2e\pi})}:=d>0,
\end{align*}
where $\frac{1}{\mu\sqrt{2e\pi}}$ is the maximum of $\tilde{b}e^{-\pi\mu^2\tilde{b}^2z\tau}\sqrt{z\tau}$ for $z\ge 0.$ Therefore, there is a $T:=T(n,\tau)>0$, such that 
\begin{align*}
\inf_{x\in[0,1]}\mathbb{E}\left(\left|u^{n,\tau}(m\tau,\kappa_n(x))\right|^2\right)\ge \frac{d}{2}e^{\pi\mu^2\tilde{b}^2m\tau}, \quad \forall \;m\tau>T.
\end{align*}
%By the linear interpolation with respect to the space variable and simple computation, we can get \eqref{sec5eq11}.\\
It follows from $\mu^2\tilde{b}^2\tau\leq1$ that $$\frac{1}{\mu\sqrt{2e\pi}}\leq \sqrt{\frac{\pi}{2e}}\int_0^{\infty}\tilde{b}e^{-\pi\mu^2\tilde{b}^2z\tau}\sqrt{z\tau}\,dz=\sqrt{\frac{\pi}{2e}}\times\frac{1}{2\pi\mu^3\tilde{b}^2\tau},$$ so we have $d\ge \frac{2\mu I^2_0}{1+\sqrt{\frac{\pi}{2e}}}$. Moreover, Lemma \ref{lemmagr} implies $\mu^2\tilde{b}^2\ge\varepsilon^2\tilde{b}^2=\frac{4\pi\zeta^2J^4_0\lambda^4}{(J^2_0+32\pi\zeta)^2.}$
Hence, the proof of the theorem is completed.
\end{proof} 
Theorem \ref{thm4.9} and Proposition \ref{prop 3.2} imply that the second moment of the numerical solution of the fully discrete scheme has sharp exponential order $C\lambda^4t_m$ for large $n$ and $t_m=m\tau$.
\begin{remark}
As for the exponential integrator method, the second moment of the numerical solution also has sharp exponential order $C\lambda^4t$ for large $n$ and $t$.
\end{remark}
\subsection{Error estimations of fully discrete scheme}\label{sec4.4}
In this subsection, we show the convergence of the fully discrete scheme. It is based on the error estimates on the fully discrete Green functions and the exact one, whose proofs are postponed to Appendix \ref{aplemma4.8}. In the sequel, we assume $n\ge 3,\; 0<\tau<1$.
%In the sequel, we always assume $n\ge 3,\; 0<\tau<1$.
\begin{lemma}\label{lemma4.8}
(\romannumeral1) Under conditions in Assumption \ref{Assumption3}, 
there is a constant $C:=C(\theta)>0$ such that for all $x\in[0,1]$,
\begin{align}\label{estimate}
\int_{0}^{\infty}\int_{0}^{1}\left|G(t,x,y)-G^{n,\tau}_2(t,x,y)\right|^2\,dydt\leq C\left(\frac{1}{n}+\sqrt{\tau}\right).
\end{align}
\iffalse
(\romannumeral1) For each fixed $\theta \in[0,\frac{1}{2})$, $n\ge 3,\; 0<\tau<1$ satisfying $n^2\tau \leq r<\frac{1}{2-4\theta}$, there is a constant $C:=C(\theta)>0$ such that for all $x\in[0,1]$,
\begin{align}\label{estimate}
\int_{0}^{\infty}\int_{0}^{1}\left|G(t,x,y)-G^{n,\tau}_2(t,x,y)\right|^2\,dydt\leq C\left(\frac{1}{n}+\sqrt{\tau}\right).
\end{align}
(\romannumeral2) For each fixed $\theta \in(\frac{1}{2},1]$, there is a constant $C:=C(\theta)>0$ such that \eqref{estimate} holds for all $x\in[0,1]$.\vspace{1 ex}\\
(\romannumeral3) For $\theta=\frac{1}{2}$, if $n^2\tau \leq \frac{1}{\epsilon}-\frac{1}{2}$ for some $0<\epsilon<\frac{1}{2}$, then \eqref{estimate} holds for all $x\in[0,1]$ with some constant $C>0$.\vspace{1 ex}\\
\fi
(\romannumeral2) Under conditions in Assumption \ref{Assumption3} (\romannumeral1) or Assumption \ref{Assumption3} (\romannumeral2) or $\theta=1$ or $\theta\in (\frac{1}{2},1)$ with $u_0\in\mathcal{C}^2([0,1])$,
%For $\theta\in[0,\frac{1}{2}),$ $n^2\tau\leq r\leq \frac{1}{2-4\theta}$, and for $\theta\in[\frac{1}{2},1),\;n\ge 3,\;0<\tau<\frac{1}{4}$ that satisfy $n^2\tau\leq \frac{1}{4(1-\theta)}\left(\frac{1}{\epsilon}+1\right)$ for some $\epsilon>0$, and for $\theta=1,\;n\ge 3,\;0<\tau<\frac{1}{4},$
%Under conditions in (\romannumeral1) (\romannumeral3) or (\romannumeral2) with $n^2\tau\leq \frac{1}{4(1-\theta)}\left(\frac{1}{\epsilon}+1\right)$ for some $\epsilon>0$,  
for any $\frac{1}{2}<\alpha <2$, there is a positive constant $C:=C(\alpha,\theta)$ such that 
\begin{align}\label{green1}
\left|\int_0^1\left(G^n(t,x,y)-G^{n,\tau}_1(t,x,y)\right)u_0(\kappa_n(y))\,dy\right|^2\leq C\tau^{\alpha-\frac{1}{2}}\left(\left[\frac{t}{\tau}\right]\tau\right)^{-\alpha}
\end{align}
for all $x\in[0,1], \;t\ge \tau>0$.
\end{lemma}
Based on Lemma \ref{lemma4.8}, we can establish the convergence theorem of the fully discrete scheme. The proof is omitted since it can be proved in a similar way as \cite{Gy99}.
\begin{thm}
 Under conditions in Lemma \ref{lemma4.8} (\romannumeral2), for every $0<\alpha<\frac{1}{4},0<\beta<\frac{1}{2},\;p\ge 1$ and for every $t\in(0,T]$ with the fixed $T>0,$ there is a constant $C:=C(\alpha,p,T)>0$ such that
\begin{align}\label{eqthm4.9}
\sup_{x\in[0,1]}\left\|u^{n,\tau}(t,x)-u(t,x)\right\|_{2p}\leq C(\tau^{-\alpha }+n^{-\beta }).
\end{align}
\end{thm}
%\begin{remark}
%If the initial data $u_0$ is a positive constant, then \eqref{eqthm4.9} holds under conditions in Lemma \ref{lemma4.8} (\romannumeral1) (\romannumeral2) (\romannumeral3).
%\end{remark}

\section{Conclusions and future aspects}\label{conclusion}
In this paper, in order to investigate the numerical schemes that could inherit the weak intermittency of the SHE and preserve the sharp exponential order of the second moment of the exact solution, we implement an approach based on the compact integral form of the numerical scheme and the detailed analysis of the discrete Green function. It is shown that the semi-discrete scheme and the fully discrete scheme are both weakly intermittent. Furthermore, both of them could preserve the sharp exponential order of the second moment of the exact solution. In fact, there are still many problems that remains to be solved. We list several possible aspects for future work:
\begin{itemize}
\item[(1)] Is there a criterion that is easy to check, to judge whether a numerical scheme can inherit the weak intermittency of the original equation?
\item[(2)] If a numerical scheme can inherit the weak intermittency of the  original equation, how to estimate the error of the Lyaponuv exponents?
\end{itemize} 

The above two problems are very challenging. Generally, the expression of the discrete Green function for a numerical scheme can not be written explicitly, so it is difficult to analyze the detailed point-wise and integral estimates of the discrete Green function. We leave these problems as open problems, and attempt to study them in our future work.

\section{Appendix}\label{sec5}
\subsection{Proof of \eqref{mild fdm}}\label{A.2}
Using notations
\begin{align*}
[U^n(t)]_k=U^n_k(t):=u^n\Big(t,\frac{k}{n}\Big),\quad [W^n(t)]_k=W^n_k(t):=\sqrt{n}\left(W\Big(t,\frac{k+1}{n}\Big)-W\Big(t,\frac{k}{n}\Big)\right),
\end{align*}
it follows from \eqref{fdm} that 
\begin{align*}
[dU^n(t)-n^2DU^n(t)dt]_k=\lambda\sqrt{n}[{\rm diag }(\sigma(U^n_0(t)),\ldots,\sigma(U^n_{n-1}(t)))\, dW^n(t)]_k,
\end{align*}
where $D=(D_{ki})$ is an $n\times n$ matrix
%where $D=(D_{ki})$ is an $n\times n$ matrix\\
\[
\left(\begin{array}{cccc}
-2 & 1 &    &1\\
 1 & -2 &\ddots &\\
    &\ddots &\ddots &1\\
  1 &         & 1&-2
  \end{array}\right).
\]
%and $W^n(t):=\left(W^n_k(t)\right)_k$ is an $n$ dimensional Wiener process.\\
The eigenvalues of $n^2D$ are $\lambda^n_j:=-4n^2\sin^2\left(\frac{j\pi}{n}\right),j=0,1,\ldots, n-1$, and the corresponding complex eigenvectors are denoted by $f_j$, whose $k$th component is $\left[f_j\right]_k:=\frac{1}{\sqrt{n}}e^{2\pi \mathbf{i}j\frac{k}{n}},j,k=0,1,\ldots, n-1$. Denote $e_j(x):=e^{2\pi \mathbf{i}jx}$. Moreover, $f_j,j=0,1,\ldots, n-1$ form an orthogonal normal basis in $\mathbb{C}^n$ (see \cite{Green01}).

Simple computation yields
\begin{align}\label{obtain1}
\left[U^n(t)\right]_k=\left[e^{n^2Dt}U^n(0)\right]_k+\lambda\sqrt{n}\left[\int_{0}^{t}e^{n^2D(t-s)}{\rm diag} (\sigma(U^n_0(s)),\ldots,\sigma(U^n_{n-1}(s)))\, dW^n(s)\right]_k.
\end{align}
Note that
\begin{align}\label{obtain2}
&\left[e^{n^2Dt}U^n(0)\right]_k=\left[\sum_{j=0}^{n-1}a_je^{\lambda^n_j t}f_j\right]_k=\sum_{j=0}^{n-1}a_j\frac{1}{\sqrt{n}}e^{\lambda^n_j t}e_j\left(\frac{k}{n}\right)\nonumber\\
=&\sum_{j=0}^{n-1}\sum_{l=0}^{n-1}u_0\left(\frac{l}{n}\right)\frac{1}{n}e^{\lambda^n_j t}\bar{e}_j\left(\frac{l}{n}\right)e_j\left(\frac{k}{n}\right)
=\int_0^1G^n\Big(t,\frac{k}{n},y\Big)u_0(\kappa_n(y))\,dy,
\end{align}
where 
\begin{align*}
G^n\Big(t,\frac{k}{n},y\Big)=\sum_{j=0}^{n-1}e^{\lambda^n_jt}e_j\left(\frac{k}{n}\right)\bar{e}_j(\kappa_n(y)),\quad
u^n(0)=\sum_{j=0}^{n-1}a_jf_j,\quad
a_j=\sum_{l=0}^{n-1}u_0\left(\frac{l}{n}\right)\frac{1}{\sqrt{n}}\bar{e}_j\left(\frac{l}{n}\right).
\end{align*}
Similarly, 
\begin{align}\label{obtain3}
&\lambda\sqrt{n}\left[\int_{0}^{t}e^{n^2D(t-s)}diag (\sigma(U^n_0(s)),\ldots,\sigma(U^n_{n-1}(s)))\, dW^n(s)\right]_k\nonumber\\
%=&\left[\sum_{j=0}^{n-1}\sqrt{n}\int_{0}^{t}e^{n^2D(t-s)}diag (\sigma(u^n_0(s)),\ldots,\sigma(u^n_{n-1}(s)))\,db_jf_j\right]_i\nonumber\\
= \,&\lambda\int_{0}^{t}\sum_{l=0}^{n-1}\sum_{j=0}^{n-1}\frac{1}{\sqrt{n}}e^{\lambda^n_j (t-s)}\sigma\left(U^n_l(s)\right)\bar{e}_j\left(\frac{l}{n}\right)e_j\left(\frac{k}{n}\right)\,dW^n_l(s)\nonumber\\
=\,&\lambda\int_0^t\int_0^1G^n\Big(t-s,\frac{k}{n},y\Big)\sigma(u^n(s,\kappa_n(y)))\,dW(s,y).
\end{align}
Combining equations \eqref{obtain1} \eqref{obtain2} and \eqref{obtain3}, we get
\begin{align}\label{variation}
u^n\Big(t,\frac{k}{n}\Big)=\int_0^1G^n\Big(t,\frac{k}{n},y\Big)u_0(\kappa_n(y))\,dy+\lambda\int_0^t\int_0^1G^n\Big(t-s,\frac{k}{n},y\Big)\sigma(u^n(s,\kappa_n(y)))\,dW(s,y).
\end{align}

We construct the continuous version of \eqref{variation} by the linear interpolation:
\begin{align*}
u^n(t,x):=u^n(t,\kappa_n(x))+(nx-n\kappa_n(x))\left[u^n\Big(t,\kappa_n(x)+\frac{1}{n}\Big)-u^n(t,\kappa_n(x))\right],\quad x\in[0,1].
\end{align*}
Denote
\begin{align*}
e^n_j(x):=e_j\left(\kappa_n(x)\right)+(nx-n\kappa_n(x))\left[e_j\Big(\kappa_n(x)+\frac{1}{n}\Big)-e_j(\kappa_n(x))\right],\quad x\in[0,1],
\end{align*}
then
\begin{align*}
G^n(t,x,y)=\sum_{j=0}^{n-1}e^{\lambda^n_j t}e^n_j(x)\bar{e}_j(\kappa_n(y)),\quad t\ge0,\;x,y\in[0,1].
\end{align*}
Obviously $u^n$ satisfies the equation
\begin{align*}
u^n(t,x)=\int_{0}^{1}G^n(t,x,y)u^n(0,\left(\kappa_n(y)\right)\,dy+\lambda\int_{0}^{t}\int_{0}^{1}G^n(t-s,x,y)\sigma\left(u^n(s,\kappa_n(y))\right)dW(s,y),
\end{align*}
almost surely for all $t\ge 0$ and $x\in [0,1]$. Hence we get \eqref{mild fdm}.

\subsection{Proof of Lemma \ref{lemma3.3}}\label{aplemma3.3}
\begin{proof}$(\romannumeral1)$  $G^n(t,x,y)$ can be rewritten as follows by expanding its real and imaginary parts:\\
If $n$ is odd,
\begin{align*}
G^n(t,x,y)=1+2\sum_{j=1}^{\left[\frac{n}{2}\right]}e^{\lambda^n_j t}\left(\varphi^n_{c,j}(x)\varphi_{c,j}(\kappa_n(y))+\varphi^n_{s,j}(x)\varphi_{s,j}(\kappa_n(y))\right).
\end{align*}
If $n$ is even,
\begin{align*}
G^n(t,x,y)=&1+2\sum_{j=1}^{\frac{n}{2}-1}e^{\lambda^n_j t}\left(\varphi^n_{c,j}(x)\varphi_{c,j}(\kappa_n(y))+\varphi^n_{s,j}(x)\varphi_{s,j}(\kappa_n(y))\right)\\
&+e^{\lambda^n_{\frac{n}{2}}t}\Big(\varphi^n_{c,\frac{n}{2}}(x)\varphi_{c,\frac{n}{2}}(\kappa_n(y))+\varphi^n_{s,\frac{n}{2}}(x)\varphi_{s,\frac{n}{2}}(\kappa_n(y))\\
&+\mathbf{i}\varphi^n_{s,\frac{n}{2}}(x)\varphi_{c,\frac{n}{2}}(\kappa_n(y))-\mathbf{i}\varphi^n_{c,\frac{n}{2}}(x)\varphi_{s,\frac{n}{2}}(\kappa_n(y))\Big)\\
=&1+2\sum_{j=1}^{\frac{n}{2}-1}e^{\lambda^n_j t}\left(\varphi^n_{c,j}(x)\varphi_{c,j}(\kappa_n(y))+\varphi^n_{s,j}(x)\varphi_{s,j}(\kappa_n(y))\right)\\&+e^{-4n^2t}\varphi^n_{c,\frac{n}{2}}(x)\varphi_{c,\frac{n}{2}}(\kappa_n(y)),
\end{align*}
where $\varphi_{c,j}(x):=\cos(2\pi jx),\varphi_{s,j}(x):=\sin(2\pi jx)$, and
\begin{align*}
&\varphi^n_{c,j}(x):=\varphi_{c,j}(\kappa_n(x))+(nx-n\kappa_n(x))\Big[\varphi_{c,j}\Big(\kappa_n(x)+\frac{1}{n}\Big)-\varphi_{c,j}(\kappa_n(x))\Big],\\
&\varphi^n_{s,j}(x):=\varphi_{s,j}(\kappa_n(x))+(nx-n\kappa_n(x))\Big[\varphi_{s,j}\Big(\kappa_n(x)+\frac{1}{n}\Big)-\varphi_{s,j}(\kappa_n(x))\Big].
\end{align*}

Let's rewrite the spectral decomposition of $G(t,x,y)$ by using notations above as follows:
\begin{align*}
G(t,x,y)&=1+2\sum_{j=1}^{\infty}e^{-4\pi^2j^2t}\cos(2\pi j(x-y))\\
&=1+2\sum_{j=1}^{\infty}e^{-4\pi^2j^2t}\left(\cos(2\pi jx)\cos(2\pi jy)+\sin(2\pi jx)\sin(2\pi jy)\right)\\
&=1+2\sum_{j=1}^{\infty}e^{-4\pi^2j^2t}\left(\varphi_{c,j}(x)\varphi_{c,j}(y)+\varphi_{s,j}(x)\varphi_{s,j}(y)\right).
\end{align*}

We first show the result by supposing $n$ is odd. It is clear that
\begin{align*}
I:=\int_{0}^{\infty}\int_{0}^{1}\left|G(t,x,y)-G^n(t,x,y)\right|^2\,dy\,dt\leq 8\sum_{k=1}^{4}(I^c_k+I^s_k),
\end{align*}
where
\begin{align*}
&I^c_1:=\int_{0}^{\infty}\int_{0}^{1}\Big|2\sum_{r=\left[\frac{n}{2}\right]+1}^{\infty}e^{-4\pi^2r^2t}\varphi_{c,r}(x)\varphi_{c,r}(y)\Big|^2\,dy\,dt,\\
&I^c_2:=\int_{0}^{\infty}\int_{0}^{1}\Big|2\sum_{r=1}^{\left[\frac{n}{2}\right]}e^{-4\pi^2r^2t}\varphi^n_{c,r}(x)\left(\varphi_{c,r}(y)-\varphi_{c,r}\left(\kappa_n(y)\right)\right)\Big|^2\,dy\,dt,\\
&I^c_3:=\int_{0}^{\infty}\int_{0}^{1}\Big|2\sum_{r=1}^{\left[\frac{n}{2}\right]}\left(e^{-4\pi^2r^2t}-e^{\lambda^n_r t}\right)\varphi^n_{c,r}(x)\varphi_{c,r}(\kappa_n(y))\Big|^2\,dy\,dt,\\
&I^c_4:=\int_{0}^{\infty}\int_{0}^{1}\Big|2\sum_{r=1}^{\left[\frac{n}{2}\right]}e^{-4\pi^2r^2t}\left(\varphi_{c,r}(x)-\varphi^n_{c,r}(x)\right)\varphi_{c,r}(y)\Big|^2\,dy\,dt,
\end{align*}
and $I^s_k,k=1,2,3,4$ are defined in a similar way via replacing $\cos(\cdot)$ by $\sin(\cdot)$.

Simple computation yields that when $n\ge 3$,
\begin{align*}
I^c_1&=4\int_{0}^{\infty}\int_{0}^{1}\sum_{r=\left[\frac{n}{2}\right]+1}^{\infty}e^{-8\pi^2r^2t}\cos^2(2\pi rx)\cos^2(2\pi ry)\,dy\,dt\\
&=2\int_{0}^{\infty}\sum_{r=\left[\frac{n}{2}\right]+1}^{\infty}e^{-8\pi^2r^2t}\cos^2(2\pi rx)\,dt\leq 2\int_{0}^{\infty}\sum_{r=\left[\frac{n}{2}\right]+1}^{\infty}e^{-8\pi^2r^2t}\,dt\\
&\leq 2\sum_{r=\left[\frac{n}{2}\right]+1}^{\infty}\frac{1}{8\pi^2r^2}\leq \frac{C}{n},\\
I^c_4&=2\int_{0}^{\infty}\sum_{r=1}^{\left[\frac{n}{2}\right]}e^{-8\pi^2r^2t}\left(\varphi_{c,r}(x)-\varphi^n_{c,r}(x)\right)^2\,dt\\
&\leq 4\int_{0}^{\infty}\sum_{r=1}^{\left[\frac{n}{2}\right]}e^{-8\pi^2r^2t}\left\{\left(\varphi_{c,r}(x)-\varphi_{c,r}(\kappa_n(x))\right)^2+\left(\varphi^n_{c,r}(\kappa_n(x))-\varphi^n_{c,r}(x)\right)^2\right\}\,dt\\
&\leq 8\int_{0}^{\infty}\sum_{r=1}^{\left[\frac{n}{2}\right]}e^{-8\pi^2 r^2 t}\times \left(\frac{2\pi r}{n}\right)^2\,dt\leq \frac{32\pi^2}{n^2}\times \frac{1}{8\pi^2}\times \left[\frac{n}{2}\right] \leq \frac{C}{n}.
\end{align*}
Moreover,
\begin{align*}
I^c_3&\leq 2\int_{0}^{\infty}\sum_{r=1}^{\left[\frac{n}{2}\right]}\left|e^{-4\pi^2 r^2t}-e^{\lambda^n_r t}\right|^2\,dt\leq 2\int_{0}^{\infty}\sum_{r=1}^{\left[\frac{n}{2}\right]}e^{-8\pi^2 r^2C^n_rt}\times \left(1-e^{-4\pi^2r^2(1-c^n_r)t}\right)^2\,dt\\
&\leq 2\int_{0}^{\infty}\sum_{r=1}^{\left[\frac{n}{2}\right]}e^{-r^2t}\times \left(4\pi^2r^2(1-c^n_r)t\right)^2\,dt\leq \frac{C}{n^4}\int_{0}^{\infty}\sum_{r=1}^{\left[\frac{n}{2}\right]}e^{-r^2t}\times r^8t^2\,dt\leq \frac{C}{n^4}\sum_{r=1}^{\left[\frac{n}{2}\right]}r^2\leq \frac{C}{n},
\end{align*}
where $c^n_r:=\sin^2\frac{r\pi}{n}/(\frac{r\pi}{n})^2\in \left[\frac{4}{\pi^2},1\right]$ for $r=1,2,\ldots,\left[\frac{n}{2}\right]$. Here we have used the inequalities $1-e^{-z}\leq z$ and $1-\frac{\sin^2 z}{z^2}\leq \frac{z^2}{3}$ for $z>0$.

In the sequel we use the notation $\tilde{G}_n(t,x,y):=1+2\sum_{r=1}^{\left[\frac{n}{2}\right]}e^{-4\pi^2r^2t}\varphi^n_{c,r}(x)\varphi_{c,r}(y)$.\\
Note that for every function $u\in \mathcal{C}^1([0,1])$,
\begin{align*}
&\int_{0}^{1}\left|u(y)-u(\kappa_n(y))\right|^2\,dy=\int_{0}^{1}\Big|\int_{\kappa_n(y)}^{y}\frac{d}{dt}u(t)\,dt\Big|^2\,dy\leq \frac{1}{n}\int_{0}^{1}\int_{\kappa_n(y)}^{y}\Big|\frac{d}{dt}u(t)\Big|^2\,dt\,dy\\
\leq&\; \frac{1}{n}\int_{0}^{1}\int_{t}^{t+\frac{1}{n}}\Big|\frac{d}{dt}u(t)\Big|^2\,dy\,dt\leq \frac{1}{n^2}\int_{0}^{1}\Big|\frac{d}{dt}u(t)\Big|^2\,dt.
\end{align*}
Thus,
\begin{align*}
I^c_2&=\int_{0}^{\infty}\int_{0}^{1}\left|\tilde{G}_n(t,x,y)-\tilde{G}_n(t,x,\kappa_n(y))\right|^2\,dy\,dt\leq \frac{1}{n^2}\int_{0}^{\infty}\int_{0}^{1}\Big|\frac{d}{dy}\tilde{G}_n(t,x,y)\Big|^2\,dy\,dt\\
&=\frac{1}{n^2}\int_{0}^{\infty}\int_{0}^{1}\Big|2\sum_{r=1}^{\left[\frac{n}{2}\right]}e^{-4\pi^2r^2t}\varphi^n_{c,r}(x)\times 2\pi r\sin(2\pi ry)\Big|^2\,dy\,dt\\
&\leq \frac{C}{n^2}\int_{0}^{\infty}\int_{0}^{1}\sum_{r=1}^{\left[\frac{n}{2}\right]}e^{-8\pi^2r^2t}\times r^2\sin^2(2\pi ry)\,dy\,dt\leq \frac{C}{n^2}\int_{0}^{\infty}\sum_{r=1}^{\left[\frac{n}{2}\right]}e^{-8\pi^2r^2t}\times r^2\,dt \leq \frac{C}{n}. 
\end{align*}
Similarly, we can get $I^s_k\leq \frac{C}{n},k=1,2,3,4$.

When $n$ is even, there is one more term denoted by $I_5$ to be estimated,
\begin{align*}
I_5\leq &C\int_{0}^{\infty}\int_{0}^{1}\Big|e^{-4n^2t}\varphi^n_{c,\frac{n}{2}}(x)\cos(2\pi \times\frac{n}{2}\times\kappa_n(y))\Big|^2\,dy\,dt\leq \frac{C}{n^2}.
\end{align*}
Combining the estimations of $I^c_k$, $I^s_k,k=1,2,3,4$ and $I_5$, we finish the proof of $(\romannumeral1)$.\\
$(\romannumeral2)$ The proof follows the same line as that of $(\romannumeral1)$. Making use of the inequality $e^{-z^2}\leq C(\alpha)z^{-\alpha}$ for $z>0,\alpha >0$, we get
\begin{align*}
\sum_{r=\left[\frac{n}{2}\right]+1}^{\infty}e^{-8\pi^2 r^2t}\leq C(\alpha)\sum_{r=\left[\frac{n}{2}\right]+1}^{\infty}r^{-2\alpha}t^{-\alpha}\leq C(\alpha)\int_{\left[\frac{n}{2}\right]}^{\infty}x^{-2\alpha}\,dx\times t^{-\alpha}\leq C(\alpha)n^{1-2\alpha}t^{-\alpha}
\end{align*}
for $\alpha>\frac{1}{2}$, and
\begin{align*}
&\sum_{r=1}^{\left[\frac{n}{2}\right]}e^{-8\pi^2 r^2t}\times r^2n^{-2}\leq C(\alpha)\sum_{r=1}^{\left[\frac{n}{2}\right]}r^{2-2\alpha}t^{-\alpha}n^{-2}\\
\leq&\; C(\alpha)\int_{0}^{\left[\frac{n}{2}\right]+1}x^{2-2\alpha}\,dx\times t^{-\alpha}n^{-2}\leq C(\alpha)n^{1-2\alpha}t^{-\alpha}
\end{align*}
for $\alpha<1$.

Similarly, we have
\begin{align*}
\sum_{r=1}^{\left[\frac{n}{2}\right]}e^{-r^2t}\times t^2r^8n^{-4}\leq C\sum_{r=1}^{\left[\frac{n}{2}\right]}r^{8-2\gamma}t^{-\gamma}n^{-4}\leq Cn^{5-2\gamma}t^{2-\gamma}=:C(\alpha)n^{1-2\alpha}t^{-\alpha}
\end{align*}
for $\alpha<2$, and 
$
e^{-8n^2t}\leq Cn^{-2\alpha}t^{-\alpha}
$
for $\alpha>0$.

Combining the terms above, we finish the proof.
\end{proof}

\subsection{Proof of \eqref{mild full}}\label{sec5.4}
It is clear that \eqref{fullscheme} is equivalent to
\begin{equation*}
\begin{aligned}
u^{n,\tau}(t_{i+1},x_j)=&\;(1-\theta \tau \Delta_n)^{-1}\left(1+(1-\theta)\tau \Delta_n\right)u^{n,\tau}(t_i,\cdot)(x_j)\\
&+(1-\theta \tau \Delta_n)^{-1}\lambda\tau\sigma(u^{n,\tau}(t_i,\cdot))\square_{n,\tau}W(t_i,\cdot)(x_j).
\end{aligned}
\end{equation*}
It is easy to check that
$$
\Delta_ne_j(x_k)=\lambda^n_je_j(x_k),\quad k,j=0,1,\ldots, n-1.
$$
Let $R_1:=(1-\theta \tau\Delta_n)^{-1}, R_2:=\left(1+(1-\theta)\tau\Delta_n\right)$, $R_{1,l}:=(1-\theta \tau\lambda^n_l)^{-1}, R_{2,l}:=\left(1+(1-\theta)\tau \lambda^n_l\right)$.\\
 By iteration, we get
\begin{equation*}
\begin{aligned}
&u^{n,\tau}(t_i,\cdot)(x_j)\\
=&\left(R_1R_2\right)^iu^{n,\tau}(t_0,\cdot)(x_j)+\lambda\sum_{k=0}^{i-1}(R_1R_2)^kR_1\tau \sigma\left(u^{n,\tau}(t_{i-1-k},\cdot)\right)\square_{n,\tau} W(t_{i-1-k},\cdot)(x_j)\\
=&\sum_{l=0}^{n-1}\sum_{k=0}^{n-1}\frac{1}{n}\left(R_{1,l}R_{2,l}\right)^iu_0\left(\frac{k}{n}\right)\bar{e}_l\left(\frac{k}{n}\right)e_l\left(\frac{j}{n}\right)\\
&+\lambda\sum_{k=0}^{i-1}\sum_{l=0}^{n-1}\sum_{q=0}^{n-1}\frac{1}{\sqrt{n}}\left(R_{1,l}R_{2,l}\right)^{i-1-k}R_{1,l}\sigma\left(u^{n,\tau}\left(t_k,\frac{q}{n}\right)\right)\left(W^n_q(t_{k+1})-W^n_q(t_k)\right)\bar{e}_l\left(\frac{q}{n}\right)e_l\left(\frac{j}{n}\right)\\
=&:\int_0^1G^{n,\tau}_1(t_i,x_j,y)u_0(\kappa_n(y))\,dy\\
&+\lambda\int_0^{t_i}\int_0^1G^{n,\tau}_{2}(t_i-\kappa_{\tau}(s)-\tau,x_j,y)\sigma(u^{n,\tau}(\kappa_{\tau}(s),\kappa_n(y)))\,dW(s,y),
\end{aligned}
\end{equation*}
where $\kappa_{\tau}(s):=\left[\frac{s}{\tau}\right]\tau$ and
\begin{align*}
&G^{n,\tau}_1(t,x_j,y)
:=\sum_{l=0}^{n-1}\left(R_{1,l}R_{2,l}\right)^{\left[\frac{t}{\tau}\right]}e_l(x_j)\bar{e}_l(\kappa_n(y)),\\
&G^{n,\tau}_2(t,x_j,y)
:=\sum_{l=0}^{n-1}\left(R_{1,l}R_{2,l}\right)^{\left[\frac{t}{\tau}\right]}R_{1,l}e_l(x_j)\bar{e}_l(\kappa_n(y)).
\end{align*}
By the linear interpolation with respect to the space variable, and denoting 
\begin{align*}
&G^{n,\tau}_1(t,x,y)
:=\sum_{l=0}^{n-1}\left(R_{1,l}R_{2,l}\right)^{\left[\frac{t}{\tau}\right]}e^n_l(x)\bar{e}_l(\kappa_n(y)),\\
&G^{n,\tau}_2(t,x,y)
:=\sum_{l=0}^{n-1}\left(R_{1,l}R_{2,l}\right)^{\left[\frac{t}{\tau}\right]}R_{1,l}e^n_l(x)\bar{e}_l(\kappa_n(y)),
\end{align*}
it is clear that $u^{n,\tau}$ satisfies the integral equation
\begin{align*}
u^{n,\tau}(t,x)=&\int_{0}^{1}G^{n,\tau}_1(t,x,y)u_0(\kappa_n(y))\,dy\\
&+\lambda\int_{0}^{t}\int_{0}^{1}G^{n,\tau}_2(t-\kappa_{\tau}(s)-\tau,x,y)\sigma \left(u^{n,\tau}(\kappa_{\tau}(s),\kappa_n(y))\right)\,dW(s,y),
\end{align*}
almost surely for every $t=i\tau, x\in [0,1]$. 
Hence we get \eqref{mild full}.

\subsection{Proof of Lemma \ref{lemma4.8}}\label{aplemma4.8}
\begin{proof}
$(\romannumeral1)$ By expanding the real and imaginary parts, the fully discrete Green functions can be written as follows,
\begin{align*}
G^{n,\tau}_1(t,x,y)
=&\;1+2\tilde{\sum}_{l}\left(R_{1,l}R_{2,l}\right)^{\left[\frac{t}{\tau}\right]}\left(\varphi^n_{c,l}(x)\varphi_{c,l}(\kappa_n(y))+\varphi^n_{s,l}(x)\varphi_{s,l}(\kappa_n(y))\right)\\
&+\left(R_{1,\frac{n}{2}}R_{2,\frac{n}{2}}\right)^{\left[\frac{t}{\tau}\right]}g_n(x,y),\\
G^{n,\tau}_2(t,x,y)=&\;1+2\tilde{\sum}_{l}\left(R_{1,l}R_{2,l}\right)^{\left[\frac{t}{\tau}\right]}R_{1,l}\left(\varphi^n_{c,l}(x)\varphi_{c,l}(\kappa_n(y))+\varphi^n_{s,l}(x)\varphi_{s,l}(\kappa_n(y))\right)\\
&+\left(R_{1,\frac{n}{2}}R_{2,\frac{n}{2}}\right)^{\left[\frac{t}{\tau}\right]}R_{1,\frac{n}{2}}g_n(x,y),
\end{align*}
where
\begin{equation*}
\begin{aligned}
g_n(x,y):=
\left\{\begin{array}{ll}
0, &n=2k+1,\vspace{1 ex}\\
\varphi^n_{c,\frac{n}{2}}(x)\varphi_{c,\frac{n}{2}}(\kappa_n(y)), &n=2k+2,\quad k=1,2,\ldots.
\end{array}
\right.
\end{aligned}
\end{equation*}

It suffices to prove $$\int_{0}^{\infty}\int_{0}^{1}\left|G^n(t,x,y)-G^{n,\tau}_2(t,x,y)\right|^2\,dy\,dt\leq C\sqrt{\tau}.$$ Let's first assume that $n$ is odd. It is obvious that
\begin{align*}
\int_{0}^{\infty}\int_{0}^{1}\left|G^n(t,x,y)-G^{n,\tau}_2(t,x,y)\right|^2\,dy\,dt
\leq\;4\sum_{j=1}^{[\frac{n}{2}]}\int_{0}^{\infty}\left|e^{\lambda^n_jt}-(R_{1,j}R_{2,j})^{\left[\frac{t}{\tau}\right]}R_{1,j}\right|^2\,dt.
\end{align*}

We use the idea and notations in Lemma \ref{lemma4.1} $(\romannumeral3)$. 
We split the set $\left\{j:1,2,\ldots,\left[\frac{n}{2}\right]\right\}$ into two parts, i.e.
\begin{align*}
\left\{j:1,2,\ldots,\left[\frac{n}{2}\right]\right\}&=\left\{j:R_{1,j}R_{2,j}\ge \frac{1}{2}\right\}\cup \left\{j:-1+\epsilon\leq R_{1,j}R_{2,j}<\frac{1}{2}\right\}\\
&=A_1\cup A_2.
\end{align*}
%where $\epsilon:=\min \left\{2-\frac{4r}{1+4\theta r},\;\frac{1}{2}\right\}>0$.\\
Recall that for $j\in A_1,$ $\frac{1}{2}<R_{2,j}<1$ and $-\lambda^n_j\tau\leq R_{3,j}\leq -2\lambda^n_j\tau$, and for $j\in A_2$, $\left|R_{1,j}R_{2,j}\right|\leq 1-\epsilon$. Moreover, we have\\
$$A_1\subset \left\{j:1\leq j\leq \frac{1}{4}\sqrt{\frac{1}{(2-\theta)\tau}}\right\} \text{ and } A_2\subset \left\{j:\frac{1}{2\pi}\sqrt{\frac{1}{(2-\theta)\tau}}<j\leq \left[\frac{n}{2}\right]\right\}.$$

Hence, for $\tau<1,$ we have
\begin{align*}
&\sum_{j\in A_1}\int_{0}^{\infty}\left|e^{\lambda^n_j t}-\left(R_{1,j}R_{2,j}\right)^{\left[\frac{t}{\tau}\right]}R_{1,j}\right|^2\,dt\\
\leq&\sum_{j\in A_1}4\int_{0}^{\infty}\left|e^{\lambda^n_j t}-\exp\left\{-R_{3,j}\frac{t}{\tau}\right\}\right|^2\,dt\\
&+\sum_{j\in A_1}4\int_{0}^{\infty}\left|\exp\left\{-R_{3,j}\frac{t}{\tau}\right\}-\exp\left\{-R_{3,j}\left[\frac{t}{\tau}\right]\right\}\right|^2\,dt\\
&+\sum_{j\in A_1}4\int_{0}^{\infty}\left|\exp\left\{-R_{3,j}\left[\frac{t}{\tau}\right]\right\}-\exp\left\{-R_{3,j}\left[\frac{t}{\tau}\right]\right\}R_{1,j}\right|^2\,dt\\
&+\sum_{j\in A_1}4\int_{0}^{\infty}\left|\exp\left\{-R_{3,j}\left[\frac{t}{\tau}\right]\right\}R_{1,j}-\left(R_{1,j}R_{2,j}\right)^{\left[\frac{t}{\tau}\right]}R_{1,j}\right|^2\,dt\\
=&:4(I_1+I_2+I_3+I_4).
\end{align*}
For the term $I_1,$
\begin{align*}
I_1\leq&\sum_{j\in A_1}\int_{0}^{\infty}e^{2\lambda^n_j t}\times \left|1-\exp\left\{\lambda^n_j \tau \left(\frac{1}{1+(1-\theta)\tau \lambda^n_j}-1\right)\frac{t}{\tau}\right\}\right|^2\,dt\\
=&\sum_{j\in A_1}\int_{0}^{\infty}e^{2\lambda^n_j t}\times \left|1-\exp\left\{-\frac{(1-\theta)\left(\tau \lambda^n_j\right)^2}{1+(1-\theta)\tau\lambda^n_j}\frac{t}{\tau}\right\}\right|^2\,dt\\
\leq& \sum_{1\leq j\leq \frac{1}{4}\sqrt{\frac{1}{(2-\theta)\tau}}}C\int_{0}^{\infty}e^{2\lambda^n_j t}\times \left(\frac{t}{\tau}\right)^2\left(\lambda^n_j \tau\right)^4\,dt
\leq \sum_{1\leq j\leq \frac{1}{4}\sqrt{\frac{1}{(2-\theta)\tau}}}C\int_{0}^{\infty}e^{-32j^2t}\times j^8\tau^2t^2\,dt\\
\leq& \sum_{1\leq j\leq \frac{1}{4}\sqrt{\frac{1}{(2-\theta)\tau}}}C\tau^2j^2\leq C\tau^2\int_{0}^{\frac{1}{4}\sqrt{\frac{1}{(2-\theta)\tau}}+1}x^2\,dx\leq C(\theta)\sqrt{\tau},
\end{align*}
where we have used the inequality $1-e^{-z}\leq z,z>0$.\\
For the term $I_2,$ by applying the mean value theorem, we get 
%and $-\lambda^n_j\tau\leq R_{3,j}\leq -2\lambda^n_j\tau$ for $j\in A_1,$ we get
\begin{align*}
I_2\leq& \sum_{j\in A_1}\int_{0}^{\infty}\exp\left\{-2R_{3,j}\left[\frac{t}{\tau}\right]\right\}\left|R_{3,j}\right|^2\,dt\\
\leq& \sum_{j\in A_1}C\int_{0}^{\infty}e^{-32j^2\tau \left[\frac{t}{\tau}\right]}\times j^4\tau^2\,dt\leq \sum_{1\leq j\leq\frac{1}{4}\sqrt{\frac{1}{(2-\theta)\tau}}}C\int_{0}^{\infty}e^{-32j^2\tau\left(\frac{t}{\tau}-1\right)}\times j^4\tau^2\,dt\\
\leq& \sum_{1\leq j\leq\frac{1}{4}\sqrt{\frac{1}{(2-\theta)\tau}}}C(\theta)\int_{0}^{\infty}e^{-32j^2t}\times j^4\tau^2\,dt\leq C(\theta)\tau^2\times \sum_{1\leq j\leq\frac{1}{4}\sqrt{\frac{1}{(2-\theta)\tau}}}j^2\leq C(\theta)\sqrt{\tau}.
\end{align*}
For terms $I_3$ and $I_4$,
\begin{align*}
I_3\leq& \sum_{j\in A_1}\int_{0}^{\infty}\exp\left\{-2R_{3,j}\left[\frac{t}{\tau}\right]\right\}\left|1-\frac{1}{1-\theta \tau \lambda^n_j}\right|^2\,dt\\
\leq& \sum_{1\leq j\leq\frac{1}{4}\sqrt{\frac{1}{(2-\theta)\tau}}}C(\theta)\int_{0}^{\infty}e^{-32j^2t}\times j^4\tau^2\,dt\leq C(\theta)\tau^2\times \sum_{1\leq j\leq\frac{1}{4}\sqrt{\frac{1}{(2-\theta)\tau}}}j^2\leq C(\theta)\sqrt{\tau},
\end{align*}
\begin{align*}
I_4\leq &\sum_{j\in A_1}\int_{0}^{\infty}\exp\left\{-2\left[\frac{t}{\tau}\right]\ln \left(1+R_{3,j}\right)\right\}\\
&\qquad\qquad\times\left|1-\exp\left\{\left[\frac{t}{\tau}\right]\left(-R_{3,j}+\ln \left(1+R_{3,j}\right)\right)\right\}\right|^2\left(\frac{1}{1-\theta \tau \lambda^n_j}\right)^2\,dt\\
\leq &\sum_{j\in A_1}\int_{0}^{\infty}\exp\left\{-2\left[\frac{t}{\tau}\right]\ln \left(1-\lambda^n_j\tau\right)\right\}\\
&\qquad\qquad\times\left|1-\exp\left\{\left[\frac{t}{\tau}\right]\left(2\lambda^n_j\tau+\ln \left(1-2\lambda^n_j\tau\right)\right)\right\}\right|^2\,dt\\
\leq &\sum_{1\leq j\leq\frac{1}{4}\sqrt{\frac{1}{(2-\theta)\tau}}}\int_{0}^{\infty}\exp\left\{2C_2\lambda^n_j \tau \left[\frac{t}{\tau}\right]\right\}\times\left|1-\exp\left\{-\left[\frac{t}{\tau}\right]C_1\left(-2\lambda^n_j\tau\right)^2\right\}\right|^2\,dt\\
\leq &\sum_{1\leq j\leq\frac{1}{4}\sqrt{\frac{1}{(2-\theta)\tau}}}C(\theta)\int_{0}^{\infty}e^{-32C_2j^2t}\times t^2j^8\tau^2\,dt\leq C(\theta)\tau^2\times \sum_{1\leq j\leq\frac{1}{4}\sqrt{\frac{1}{(2-\theta)\tau}}}j^2\leq C(\theta)\sqrt{\tau},
\end{align*}
where we have used the fact that $z:=-\lambda^n_j\tau\in \left(0,\frac{\pi^2}{4(2-\theta)}\right]$, $j=1,2,\ldots\,\left[\frac{n}{2}\right]$ because of $j^2\tau\leq \frac{1}{16(2-\theta)}$. For such $z$, we have $-C_1z^2\leq-z+\ln (1+z)\leq 0$ and $\ln (1+z)\ge C_2z$ for some $C_1,C_2>0$.

It remains to prove
\begin{align*}
I_5:=\sum_{j\in A_2}\int_{0}^{\infty}e^{2\lambda^n_j t}\,dt
\leq \sum_{\frac{1}{2\pi}\sqrt{\frac{1}{(2-\theta)\tau}}<j\leq \left[\frac{n}{2}\right]}\int_{0}^{\infty}e^{-32j^2t}\,dt\leq \sum_{\frac{1}{2\pi}\sqrt{\frac{1}{(2-\theta)\tau}}<j\leq \left[\frac{n}{2}\right]}\frac{C}{j^2}\leq C(\theta)\sqrt{\tau}
\end{align*}
and
\begin{align*}
&I_6:=\sum_{j\in A_2}\int_{0}^{\infty}\left(R_{1,j}R_{2,j}\right)^{2\left[\frac{t}{\tau}\right]}R^2_{1,j}\,dt
\leq \sum_{j\in A_2}\int_{0}^{\infty}(1-\epsilon)^{2\left[\frac{t}{\tau}\right]}\times (1+16\theta j^2\tau)^{-2}\,dt\\
\leq &\int_{\frac{1}{2\pi}\sqrt{\frac{1}{(2-\theta)\tau}}}^{\left[\frac{n}{2}\right]}\int_0^{\infty}(1-\epsilon)^{2\left[\frac{t}{\tau}\right]}\times (1+16\theta x^2\tau)^{-2}\,dt\,dx\quad \left(\text{let }r=\frac{t}{\tau},\;y=x\sqrt{\tau}\right)\\
= &\;\sqrt{\tau}\int_{\frac{1}{2\pi}\sqrt{\frac{1}{2-\theta}}}^{\left[\frac{n}{2}\right]\sqrt{\tau}}\int_0^{\infty}(1-\epsilon)^{2[r]}\times (1+16\theta y^2)^{-2}\,dr\,dy\\
= &\;\sqrt{\tau}\sum_{k=0}^{\infty}(1-\epsilon)^{2k}\int_{\frac{1}{2\pi}\sqrt{\frac{1}{2-\theta}}}^{\left[\frac{n}{2}\right]\sqrt{\tau}}(1+16\theta y^2)^{-2}\,dy
\leq C(\theta)\sqrt{\tau},
\end{align*}
where in the last line we use the same analysis as in Lemma \ref{lemma4.1} $(\romannumeral3)$.

When $n$ is even, what we need to prove is the difference of the term of $j=\frac{n}{2}$ in the expansion of $G^n$ and $G^{n,\tau}_2$, i.e.,
\begin{align*}
&\int_{0}^{\infty}\int_{0}^{1}\left|\left(e^{-4n^2}-\left(R_{1,\frac{n}{2}}R_{2,\frac{n}{2}}\right)^{\left[\frac{t}{\tau}\right]}R_{1,\frac{n}{2}}\right)g_n(x,y)\right|^2\,dy\,dt\\
\leq&\;\int_{0}^{\infty}\left|e^{-4n^2}-\left(R_{1,\frac{n}{2}}R_{2,\frac{n}{2}}\right)^{\left[\frac{t}{\tau}\right]}R_{1,\frac{n}{2}}\right|^2\,dt\\
\leq&\; 2\int_{0}^{\infty}e^{-8n^2t}\,dt+2\int_{0}^{\infty}\left(1-\epsilon\right)^{2\left[\frac{t}{\tau}\right]}(1+4\theta n^2\tau)^{-2}\,dt\\
\leq&\;\frac{C}{n^2}+C\int_{0}^{\infty}\left(1-\epsilon\right)^{2\left[\frac{t}{\tau}\right]}\,dt\leq \frac{C}{n^2}+C\tau\sum_{k=0}^{\infty}\left(1-\epsilon\right)^{2k}\leq\frac{C}{n^2}+C\tau.
\end{align*}
Hence the proof of $(\romannumeral1)$ is completed.\\
$(\romannumeral2)$
Let's suppose that $n$ is odd since the even case can be proved similarly.

 %\textit{Case 1: $\theta\in[0,\frac{1}{2})$.} Under conditions in $(\romannumeral1)$,
 For the cases of $\theta\in[0,\frac{1}{2}]$ and $\theta=1,$ it suffices to prove for any $\frac{1}{2}<\alpha<2$,
\begin{align}\label{prove}
\int_{0}^{1}\left|G^n(t,x,y)-G^{n,\tau}_1(t,x,y)\right|^2\,dy\leq C\tau^{\alpha-\frac{1}{2}}\left(\left[\frac{t}{\tau}\right]\tau\right)^{-\alpha}
\end{align}
with some $C:=C(\alpha,\theta)>0$ for all $x\in[0,1],t\ge \tau$.
The proof of \eqref{prove} follows the same order of the terms $I_i,i=1,2,\ldots,6$ in $(\romannumeral1)$ after replacing $G^{n,\tau}_2$ by $G^{n,\tau}_1$ and removing the integral with $t$, we still denote them by $I_i,i=1,2,\ldots,6.$\\
 When $t\ge \tau$, we have for $\alpha<2$,
\begin{align*}
I_1=&\sum_{j\in A_1}e^{2\lambda^n_j t}\times \left|1-\exp\left\{\lambda^n_j \tau\left(\frac{1}{1+(1-\theta)\tau \lambda^n_j}-1\right)\frac{t}{\tau}\right\}\right|^2\\
\leq& \sum_{1\leq j\leq \frac{1}{4}\sqrt{\frac{1}{(2-\theta)\tau}}}Ce^{2\lambda^n_j \left[\frac{t}{\tau}\right]\tau}\times j^8\tau^4\left(\left[\frac{t}{\tau}\right]+1\right)^2\\
\leq &\sum_{1\leq j\leq \frac{1}{4}\sqrt{\frac{1}{(2-\theta)\tau}}}C(\gamma)\left(j^2\tau \left[\frac{t}{\tau}\right]\right)^{-\gamma}j^8\tau^4\left[\frac{t}{\tau}\right]^2\quad (\text{let }\alpha=\gamma-2)\\
\leq& C(\alpha,\theta)\tau^{\alpha-\frac{1}{2}}\left(\left[\frac{t}{\tau}\right]\tau\right)^{-\alpha},
\end{align*}
and for $\alpha<2$,
\begin{align*}
I_2=&\sum_{j\in A_1}\exp\left\{-2R_{3,j}\left[\frac{t}{\tau}\right]\right\}\left|R_{3,j}\right|^2
\leq \sum_{1\leq j\leq \frac{1}{4}\sqrt{\frac{1}{(2-\theta)\tau}}}C(\alpha)\left(j^2\tau \left[\frac{t}{\tau}\right]\right)^{-\alpha}j^4\tau^2\\
\leq& C(\alpha,\theta)\tau^{\alpha-\frac{1}{2}}\left(\left[\frac{t}{\tau}\right]\tau\right)^{-\alpha}.
\end{align*}
Terms $I_3$ and $I_4$ can be estimated in a similar way.\\
By using the inequality $e^{-z^2}\leq C(\beta)z^{-\beta},z>0,\beta>0,$ we get for $\alpha>\frac{1}{2}$,
\begin{align*}
&I_5=\sum_{\frac{1}{2\pi}\sqrt{\frac{1}{(2-\theta)\tau}}<j\leq \left[\frac{n}{2}\right]}e^{2\lambda^n_jt}
%\leq \sum_{\frac{1}{2\pi}\sqrt{\frac{1}{(2-\theta)\tau}}<j\leq \left[\frac{n}{2}\right]}e^{-32j^2t}
\leq \sum_{\frac{1}{2\pi}\sqrt{\frac{1}{(2-\theta)\tau}}<j\leq \left[\frac{n}{2}\right]}e^{-32j^2\left[\frac{t}{\tau}\right]\tau}
\leq \sum_{\frac{1}{2\pi}\sqrt{\frac{1}{(2-\theta)\tau}}<j\leq \left[\frac{n}{2}\right]}C(\alpha)\left(j^2\left[\frac{t}{\tau}\right]\tau \right)^{-\alpha}\\
&\leq C(\alpha)\int_{\frac{1}{2\pi}\sqrt{\frac{1}{(2-\theta)\tau}}}^{\left[\frac{n}{2}\right]}x^{-2\alpha}\,dx\times \left(\left[\frac{t}{\tau}\right]\tau \right)^{-\alpha}\leq C(\alpha,\theta)\tau^{\alpha-\frac{1}{2}}\left(\left[\frac{t}{\tau}\right]\tau\right)^{-\alpha},
\end{align*}
and the remaining term $I_6$ can be estimated as follows.\\
For $\theta\in[0,\frac{1}{2}]$, because $n^2\tau$ is bounded, we have for $\alpha>0$, 
%because of $\left|R_{1,j}R_{2,j}\right|\leq 1-\epsilon$ for $j\in A_2$, we have
\begin{align*}
\sum_{j\in A_2}\left(R_{1,j}R_{2,j}\right)^{2\left[\frac{t}{\tau}\right]}
%\leq \sum_{j\in A_2}(1-\epsilon)^{2\left[\frac{t}{\tau}\right]}
\leq \frac{1}{\sqrt{\tau}}\int_{\frac{1}{2\pi}\sqrt{\frac{1}{2-\theta}}}^{\left[\frac{n}{2}\right]\sqrt{\tau}}(1-\epsilon)^{2\left[\frac{t}{\tau}\right]}\,dy\leq \frac{C(\theta)}{\sqrt{\tau}}e^{-2\left[\frac{t}{\tau}\right]\ln(1-\epsilon)^{-1}}\leq \frac{C(\theta,\alpha)}{\sqrt{\tau}}\left[\frac{t}{\tau}\right]^{-\alpha}.
\end{align*}
%\textit{Case 2: $\theta=\frac{1}{2}$.} Under conditions in $(\romannumeral3)$, \eqref{green1} can be proved similarly as in \textit{Case 1}.\\
For $\theta=1,$ we have for $\alpha>0,\;t\ge \tau$,
\begin{align*}
&\sum_{\frac{1}{2\pi}\sqrt{\frac{1}{\tau}}<j\leq \left[\frac{n}{2}\right]}\left(1-\tau\lambda^n_j\right)^{-2\left[\frac{t}{\tau}\right]}\leq \sum_{\frac{1}{2\pi}\sqrt{\frac{1}{\tau}}<j\leq \left[\frac{n}{2}\right]}(1+16j^2\tau)^{-2\left[\frac{t}{\tau}\right]}
\leq \frac{1}{\sqrt{\tau}}\int_{\frac{1}{2\pi}}^{\infty}(1+16y^2)^{-2\left[\frac{t}{\tau}\right]}\,dy\\
&\leq \frac{1}{\sqrt{\tau}}\left(1+\frac{4}{\pi^2}\right)^{-\left[\frac{t}{\tau}\right]}\times \int_{\frac{1}{2\pi}}^{\infty}(1+16y^2)^{-\left[\frac{t}{\tau}\right]}\,dy
\leq C(\alpha)\frac{1}{\sqrt{\tau}}\left[\frac{t}{\tau}\right]^{-\alpha}.
\end{align*}

For the case of $\theta\in(\frac{1}{2},1),$
%\textit{Case 4: $\theta\in(\frac{1}{2},1).$} 
we split the integral $$\int_{0}^{1}\left(G^n(t,x,y)-G^{n,\tau}_1(t,x,y)\right)u_0(\kappa_n(y))\,dy$$ into two parts, i.e.,
\begin{align*}
&\int_{0}^{1}\left(G^n(t,x,y)-G^{n,\tau}_1(t,x,y)\right)u_0(\kappa_n(y))\,dy\\
=&\int_{0}^{1}\left(G^n(t,x,y)-G^{n,\tau}_1(t,x,y)\right)u_0(y)\,dy\\
&+\int_{0}^{1}\left(G^n(t,x,y)-G^{n,\tau}_1(t,x,y)\right)\left(\left(u_0(\kappa_n(y))-u_0(y)\right)\right)\,dy\\
=&:Q_1+Q_2.
\end{align*}

We only consider the space grid points $\kappa_n(x)$ because the result of other points can be obtained by the inequality $(a+b)^2\leq 2a^2+2b^2$. Taking $x=\frac{i}{n},i=0,1,\ldots,n-1,$ then we have $\varphi^n_{c,j}(x)=\varphi_{c,j}(x)=\cos(2\pi jx),\;\varphi^n_{s,j}(x)=\varphi_{s,j}(x)=\sin(2\pi jx).$ So 
\begin{align*}
Q_1&=\int_0^1 2\sum_{j=1}^{\left[\frac{n}{2}\right]}\left(e^{\lambda^n_jt}-\left(R_{1,j}R_{2,j}\right)^{\left[\frac{t}{\tau}\right]}\right)\left(\varphi^n_{c,j}(x)\varphi_{c,j}(y)+\varphi^n_{s,j}(x)\varphi_{s,j}(y)\right)u_0(y)\,dy\\
&=2\sum_{j=1}^{\left[\frac{n}{2}\right]}\left(e^{\lambda^n_jt}-\left(R_{1,j}R_{2,j}\right)^{\left[\frac{t}{\tau}\right]}\right)\int_{0}^{1}\cos\left(2\pi j(x-y)\right)u_0(y)\,dy.
\end{align*}
Suppose that the initial data $u_0\in\mathcal{C}^2([0,1])$, then by the integration by parts formula and $u_0(0)=u_0(1)$, we get for $j=1,\ldots,\left[\frac{n}{2}\right],$
\begin{align*}
&\int_{0}^{1}\cos\left(2\pi j(x-y)\right)u_0(y)\,dy=-\frac{1}{2\pi j}\int_0^1 u_0(y)\,d\sin(2\pi j(x-y))\\
=&\;-\frac{1}{2\pi j}\left(u_0(1)\sin(2\pi j(x-1))-u_0(0)\sin(2\pi jx)\right)+\frac{1}{2\pi j}\int_0^1\sin(2\pi j(x-y))u'_0(y)\,dy\\
=&\;\frac{1}{(2\pi j)^2}\int_0^1u'_0(y)\,d\cos(2\pi j(x-y))\\
=&\;\frac{1}{(2\pi j)^2}\left(u'_0(1)\cos(2\pi j(x-1))-u'_0(0)\cos(2\pi jx)-\int_0^1\cos(2\pi j(x-y))u''_0(y)\,dy\right)
\leq \frac{C}{j^2}.
\end{align*}
Hence,
\begin{align*}
Q_1&\leq C\sum_{j=1}^{\left[\frac{n}{2}\right]}\left|e^{\lambda^n_jt}-\left(R_{1,j}R_{2,j}\right)^{\left[\frac{t}{\tau}\right]}\right|\frac{1}{j^2}\\
&\leq \sum_{j\in A_1}C\left|e^{\lambda^n_j t}-\left(R_{1,j}R_{2,j}\right)^{\left[\frac{t}{\tau}\right]}\right|\frac{1}{j^2}+\sum_{j\in A_2}C\left|e^{\lambda^n_j t}-\left(R_{1,j}R_{2,j}\right)^{\left[\frac{t}{\tau}\right]}\right|\frac{1}{j^2}
=:Q_{11}+Q_{12}.
\end{align*}
For the term $Q_{12},$
\begin{align*}
Q_{12}&\leq \sum_{j\in A_2}Ce^{\lambda^n_j t}\times \frac{1}{j^2}+\sum_{j\in A_2}C\left|R_{1,j}R_{2,j}\right|^{\left[\frac{t}{\tau}\right]}\frac{1}{j^2}\\
&\leq\sum_{\frac{1}{2\pi}\sqrt{\frac{1}{(2-\theta)\tau}}<j\leq \left[\frac{n}{2}\right]}Ce^{-16j^2t}\times \frac{1}{j^2}+\sum_{\frac{1}{2\pi}\sqrt{\frac{1}{(2-\theta)\tau}}<j\leq \left[\frac{n}{2}\right]}C(1-\epsilon)^{\left[\frac{t}{\tau}\right]}\times \frac{1}{j^2}\\
&\leq C\left(e^{-\frac{4}{\pi^2(2-\theta)}\times \left[\frac{t}{\tau}\right]}+(1-\epsilon)^{\left[\frac{t}{\tau}\right]}\right)\int_{\frac{1}{2\pi}\sqrt{\frac{1}{(2-\theta)\tau}}}^{\left[\frac{n}{2}\right]}\frac{1}{x^2}\,dx
\leq C(\alpha,\theta)\left[\frac{t}{\tau}\right]^{-\alpha}\sqrt{\tau},
\end{align*}
where $\alpha>0$, and in the last line we use the inequality $e^{-z^2}\leq C(\alpha)z^{-\alpha},\alpha>0,z>0$.\\
For the term $Q_{11},$ by Cauchy-Schwarz inequality, we have for $\frac{1}{2}<\alpha<2,$
\begin{align*}
Q_{11}^2&\leq \sum_{j\in A_1}C\left|e^{\lambda^n_j t}-\left(R_{1,j}R_{2,j}\right)^{\left[\frac{t}{\tau}\right]}\right|^2\times \sum_{j\in A_1}\frac{1}{j^4}
\leq \sum_{j\in A_1}C\left|e^{\lambda^n_j t}-\left(\frac{1+(1-\theta)\tau\lambda^n_j}{1-\theta \tau\lambda^n_j}\right)^{\left[\frac{t}{\tau}\right]}\right|^2\\
&\leq C(\alpha,\theta)\tau^{\alpha-\frac{1}{2}}\left(\left[\frac{t}{\tau}\right]\tau\right)^{-\alpha},
\end{align*}
where the last step can be obtained as before.

For the term $Q_{2},$ since $u_0\in\mathcal{C}^2([0,1]),$ we have
\begin{align*}
Q_2^2
&\leq \frac{C}{n^2}\int_0^1\left|G^n(t,x,y)-G^{n,\tau}_1(t,x,y)\right|^2\,dy
\leq \frac{C}{n^2}\sum_{j=1}^{\left[\frac{n}{2}\right]}\left|e^{\lambda^n_jt}-\left(R_{1,j}R_{2,j}\right)^{\left[\frac{t}{\tau}\right]}\right|^2\\
&\leq \sum_{j\in A_1}\frac{C}{n^2}\left|e^{\lambda^n_j t}-\left(R_{1,j}R_{2,j}\right)^{\left[\frac{t}{\tau}\right]}\right|^2
+\sum_{j\in A_2}\frac{C}{n^2}\left|e^{\lambda^n_j t}-\left(R_{1,j}R_{2,j}\right)^{\left[\frac{t}{\tau}\right]}\right|^2=:Q_{21}+Q_{22},
\end{align*}
where 
\begin{align*}
Q_{22}\leq \frac{C}{n^2}\times \left[\frac{n}{2}\right]\times\left(e^{-\frac{8}{\pi^2(2-\theta)}\times \left[\frac{t}{\tau}\right]}+(1-\epsilon)^{2\left[\frac{t}{\tau}\right]}\right)
\leq\frac{C(\theta,\alpha)}{n}\left[\frac{t}{\alpha}\right]^{-\alpha}
\end{align*}
for $\alpha>0$, and $Q_{21}\leq C(\alpha,\theta)\tau^{\alpha-\frac{1}{2}}\left(\left[\frac{t}{\tau}\right]\tau\right)^{-\alpha}$ for $\frac{1}{2}<\alpha<2$.

Combining the above terms, we have for any $\frac{1}{2}<\alpha<2$,
\begin{align*}
\left|\int_{0}^{1}\left(G^n(t,x,y)-G^{n,\tau}_1(t,x,y)\right)u_0(\kappa_n(y))\,dy\right|^2
\leq 2Q_1^2+2Q_2^2\leq C(\alpha,\theta)\tau^{\alpha-\frac{1}{2}}\left(\left[\frac{t}{\tau}\right]\tau\right)^{-\alpha}.
\end{align*}
The proof is finished.
\end{proof}

\bibliographystyle{plain}
\bibliography{intermittencybib}

 \end{document}